\documentclass[11pt,a4paper,reqno]{amsart}
\usepackage[left=30mm, right=30mm, bottom=25mm, top=25mm,includefoot]{geometry}
\usepackage{MnSymbol}
\usepackage{mathrsfs}
\usepackage{amsmath,amsthm}
\usepackage{tikz}
\usetikzlibrary{
  knots,
  hobby,
  decorations.pathreplacing,
  shapes.geometric,
  calc,
  decorations.markings
}
\usepgfmodule{decorations}
\usepackage{float}
\usepackage{setspace}
\usepackage{tikz-cd}
\usetikzlibrary{braids}
\usepackage{wrapfig}
\usepackage{lipsum} 
\usepackage{marginnote}
\definecolor{invisible}{RGB}{245,245,245}
\usepackage[tableposition=below]{caption}

\usetikzlibrary{matrix}

\usepackage{array}

\usepackage{graphicx}
\usepackage[all]{xy}

\theoremstyle{plain}
\newtheorem{Theorem}{Theorem}[section]

\newtheorem{conjecture}[Theorem]{Conjecture}
\newtheorem{Lemma}[Theorem]{Lemma}

\newtheorem{Proposition}[Theorem]{Proposition}
\newtheorem{Corollary}[Theorem]{Corollary}

\theoremstyle{definition}

\newtheorem{definition}[Theorem]{Definition}

\numberwithin{equation}{section}
\makeatletter
\newcommand*\PrintSkips[1]{%
  \typeout{In #1:}%
  \typeout{\@spaces above: \the\abovecaptionskip}%
  \typeout{\@spaces below: \the\belowcaptionskip}%
}
\makeatother

\newcommand{\X}{\mathrm{X}}
\newcommand{\Y}{\mathrm{Y}}
\newcommand{\M}{\mathcal{M}}
\newcommand{\D}{\mathcal{D}}
\newcommand{\U}{\mathcal{U}}
\newcommand{\C}{\mathcal{C}}

\renewcommand{\L}{\mathcal{L}}
\renewcommand{\P}{\mathcal{P}}

\newcommand{\V}{\mathcal{V}}
\newcommand{\Q}{\mathcal{Q}}
\newcommand{\A}{\mathcal{A}}
\newcommand{\T}{\mathcal{T}}

\newcommand{\B}{\mathcal{B}}

\newcommand{\p}{\pi}
\renewcommand{\a}{\alpha}

\begin{document}

\title{New classes of minimal knot diagrams}
\author{Ilya Alekseev}
\thanks{This work was supported by the Russian Science Foundation (project no. 19-11-00151)}
\address{Leonhard Euler International Mathematical Institute in Saint Petersburg, 14th Line 29B, Vasilyevsky Island, Saint Petersburg, 199178, Russia}
\email{ilyaalekseev@yahoo.com}

\begin{abstract}
We describe a new class of minimal link diagrams. This class includes certain alternating diagrams, the standard diagrams of all torus links, and numerous homogeneous diagrams whose minimality has not been proven before. Besides, we describe a new larger class of link diagrams with the least number of Seifert circles among all diagrams of a given link. Our approach refers to the Morton--Franks--Williams inequality.
\end{abstract}

\maketitle

\ \vspace{-1.2cm}

\section{Introduction}

The central concepts of the present paper are minimal link diagrams and the crossing number of links. 
A diagram $\D$ of a link~$\L$ is said to be {\it minimal} if $\D$ has the least number of crossings among all diagrams of $\L$. The number of crossings of a minimal diagram of~$\L$ is called the {\it crossing number} of $\L$.

In this paper, we assume that all link diagrams are oriented, and for standard definitions, we mostly use the conventions of \cite{CromwellBook}. 

At this moment, several remarkable phenomena were discovered concerning minimal diagrams. In particular, certain easy to verify visual properties of link diagrams guarantee the minimality. The examples include reduced alternating diagrams, adequate diagrams, reduced Montesinos diagrams, and closures of certain positive braid diagrams due to~J.~Gonz\`{a}lez-Meneses and P. M. G. Manch\`{o}n. 

In the present paper, we introduce {\it locally twisted} link diagrams, which extend the above list. To~state our main result, we require an additional notion. In \cite{Cromwell89}, P. R. Cromwell introduces the class of {\it homogeneous} link diagrams, which contains both alternating and positive ones, and proposes studying the crossing number of the corresponding links. We~provide a sufficient condition of minimality for homogeneous diagrams.

\begin{Theorem}\label{Theorem1}
All locally twisted homogeneous link diagrams are minimal.
\end{Theorem}

To define locally twisted diagrams, we introduce link templates. To be precise, these diagrams involve specific templates, referred to as {\it knitted}.
The idea of the templates is that one can obtain each link diagram by arranging several braid diagrams on the plane and connecting them appropriately. 
In this paper, we consider homogeneous braid diagrams.
One divides each such diagram into {\it uniform layers}. 
Roughly speaking, a link diagram is locally twisted if it admits a knitted template such that each of the uniform layers of the corresponding braid diagrams is twisted enough in a certain sense.

To prove Theorem~\ref{Theorem1}, we establish 
a more general result (Theorem \ref{Theorem2}), which refers to the notion of Seifert circles.

We say that a diagram $\D$ of a link~$\L$ is {\it optimal} if $\D$ has the least number of Seifert circles among all diagrams of $\L$. The number of Seifert circles of an optimal diagram of $\L$ equals the braid index of $\L$ (see \cite[Theorem 3]{Yamada}).

A well-known inequality referred to as the Morton--Franks--Williams inequality relates the breadth of the skein polynomial of the link given by a diagram $\D$ with the number of Seifert circles of $\D$. If the Morton--Franks--Williams inequality is sharp for $\D$, then $\D$ is optimal.

\begin{Theorem}\label{Theorem2}
The Morton--Franks--Williams inequality is sharp for all locally twisted link diagrams.
\end{Theorem}

\begin{Corollary}\label{Corollary1}
All locally twisted link diagrams are optimal.
\end{Corollary}

We deduce several results related to Theorem \ref{Theorem1}. We say that a link is {\it locally twisted} (resp.\ {\it locally twisted homogeneous}) if it admits a locally twisted (resp.\ locally twisted homogeneous) diagram.

A well-known conjecture states that the crossing number of all connected sums of two links equals the sum of the crossing numbers of these links (see \cite[Problem~1.65]{K}, \cite{Lackenby}, and \cite[p.\ 69]{Ad94}). Equivalently, all connected sums of minimal link diagrams are minimal. This conjecture is still open. Since all connected sums of adequate diagrams are adequate, the conjecture holds for adequate links. It is easy to check that all connected sums of locally twisted homogeneous links admit locally twisted homogeneous diagrams. Therefore, Theorem~\ref{Theorem1} implies the following result.

\begin{Corollary}\label{Corollary4}
The conjecture that the crossing number is additive under connected sums holds for all locally twisted homogeneous links in the sense that all connected sums of locally twisted homogeneous diagrams are minimal.
\end{Corollary}

Following~\cite{Malyutin}, we say that a minimal diagram $\D$ is~{\it~$1$-regular} if for any diagram~$\D^\prime$, the crossing numbers of links determined by connected sums of $\D$ and $\D^\prime$ are bounded from below by the number of crossings of $\D$. 
If the conjecture on the additivity of crossing number holds,
then any minimal link diagram is~$1$-regular. By using properties of the Kauffman polynomial, one can show that all adequate diagrams are~$1$-regular (see~\cite[Theorem~1]{B20}). Also, all minimal diagrams of torus links are~$1$-regular (see \cite[Theorem~3.8]{AOCN}). In~\cite[Theorem~3.8]{AOCN}, the author gives a sufficient condition on a link diagram to be~$1$-regular. It follows from~\cite[Corollary~4.1]{Cromwell89} that all optimal homogeneous diagrams satisfy this condition. Therefore, Corollary~\ref{Corollary1} implies the following result.

\begin{Corollary}\label{Corollary5}
All locally twisted homogeneous diagrams are $1$-regular.
\end{Corollary}

A braid diagram is said to be {\it minimal} if it has the least number of crossings among all braid diagrams that represent the same braid. A problem due to J. Stallings (see \cite[Problem 1.8]{K}) asks whether words in the standard Artin generators (and their inverses) corresponding to minimal braid diagrams are closed under end extension (replacing a final letter $s$ by $ss$). We propose a natural analog of this conjecture for minimal link diagrams. 

The link diagram transformations of the form
\begin{tikzpicture}[rotate=90, scale=0.9, baseline=1]
\draw[->] (0.3,0.3)--(0,0);
\draw[thick] (0.3,0.3)--(0,0);
\draw[thick] (0.2,0.1)--(0.3,0);
\draw[->] (0.2,0.1)--(0.3,0);
\draw[thick] (0,0.3)--(0.1,0.2);
\end{tikzpicture}
$\mapsto$
\begin{tikzpicture}[rotate=90, scale=0.9, baseline=1]
\draw[thick] (0.3,0.3)--(0,0);
\draw[->] (0.3,0.3)--(0,0);
\draw[thick] (0.2,0.1)--(0.3,0);
\draw[->] (0.2,0.1)--(0.3,0);
\draw[thick] (0,0.3)--(0.1,0.2);
\draw[thick] (0.3,0.6)--(0,0.3);
\draw[thick] (0.2,0.4)--(0.3,0.3);
\draw[thick] (0,0.6)--(0.1,0.5);
\end{tikzpicture}
and
\begin{tikzpicture}[rotate=90, scale=0.9, baseline=1]
\draw[->] (0,0.3)--(0.3,0);
\draw[thick] (0,0.3)--(0.3,0);
\draw[thick] (0.1,0.1)--(0,0);
\draw[->] (0.1,0.1)--(0,0);
\draw[thick] (0.3,0.3)--(0.2,0.2);
\end{tikzpicture}
$\mapsto$
\begin{tikzpicture}[rotate=90, scale=0.9, baseline=1]
\draw[->] (0,0.3)--(0.3,0);
\draw[thick] (0,0.3)--(0.3,0);
\draw[thick] (0.1,0.1)--(0,0);
\draw[->] (0.1,0.1)--(0,0);
\draw[thick] (0.3,0.3)--(0.2,0.2);
\draw[thick] (0.3,0.3)--(0,0.6);
\draw[thick] (0.1,0.4)--(0,0.3);
\draw[thick] (0.3,0.6)--(0.2,0.5);
\end{tikzpicture}
are referred to as the~{\it~doubling of a crossing}.

\begin{conjecture}\label{SMBC1}
Let~$\D$ be an oriented link diagram. Let $\D^\prime$ be a diagram obtained from~$\D$ by the doubling of a crossing. If~$\D$ is minimal, then~$\D^\prime$ is minimal too.
\end{conjecture}

It follows by definition that all diagrams obtained from adequate ones by the doubling of a crossing are adequate. Therefore, Conjecture~\ref{SMBC1} holds for them. It is easy to check that all diagrams obtained from locally twisted homogeneous ones by the doubling of a crossing are locally twisted homogeneous. Therefore, Theorem~\ref{Theorem1} implies the following result.

\begin{Corollary}
{\normalfont Conjecture~\ref{SMBC1}} holds for all locally twisted homogeneous link diagrams in the sense that all link diagrams obtained from locally twisted homogeneous ones by the doubling of a crossing are minimal.
\end{Corollary}

\subsection{Related results}

We list several results related to Theorem \ref{Theorem2}.

If a link diagram $\D$ has two Seifert circles that share precisely one crossing, then $\D$ is not optimal. Namely, one can apply a transformation that reduces the number of Seifert circles. This transformation is referred to as the {\it Murasugi--Przytycki move} (see~\cite[Definition~2.1]{Index}). In \cite[Theorem~1.1]{BIRAL}, the authors show that
if an alternating diagram~$\D$ has no pair of Seifert circles that share precisely one crossing, 
then the Morton--Franks--Williams inequality is sharp for $\D$. Thus, they characterize those alternating diagrams 
that are optimal. Many, but not all, optimal alternating diagrams admit flypes (see below) that transform them into locally twisted ones.

All prime knots up to $10$ crossings except $9_{42}$, $9_{49}$, $10_{132}$, $10_{150}$, and $10_{156}$ admit diagrams for which the Morton--Franks--Williams inequality is sharp 
(see \cite[p.\ 174]{Ad94}).
The~same is true for the closures of all positive braids with
at most $3$ strands (see~\cite[Proposition~3.1]{Nakamura}), fibered alternating links (see~\cite[Theorem~A]{MurasugiA}), and rational links (see~\cite[Theorem~B]{MurasugiA}).

Another class of link diagrams for which the Morton--Franks--Williams inequality is sharp is that of the closures of all positive braid diagrams with the full twist
(see \cite[Corollary~2.4]{FW}, \cite[Theorem~10.5.1]{CromwellBook}, \cite[Corollary~4.5]{GM}, \cite[Theorem~1.3]{Kalman}, and~\cite[Corollary~1.4]{Feller}). 
For example, it follows by definition that standard diagrams of all torus links lie within this class. 
In \cite{Nak20}, the author generalizes the braid closure construction and introduces {\it knitted diagrams}, which are similar to the knitted templates defined below.
It follows from \cite[Theorem~2.4]{Nak20} that 
the Morton--Franks--Williams inequality is sharp for all positive knitted diagrams with the full twists. 
It turns out that a positive link diagram is knitted if and only if it is locally twisted.

We recall the definition of certain braid diagrams introduced by J. Gonz\`{a}lez-Meneses and P. M. G. Manch\`{o}n in \cite{GM}.
We refer to these braid diagrams as {\it GMM} ones.
By~definition, 
a braid diagram is a GMM diagram if and only if the corresponding word in the alphabet of the standard Artin generators $\{\sigma_1,\sigma_2, \ldots \}$ is obtained from the empty word by a finite sequence of transformations of the following types:
\begin{enumerate}
\item for some $i \in \{1, 2, \ldots, n-1\}$, inserting $\sigma_i\sigma_i$;
\item for some $i \in \{1, 2, \ldots, n-1\}$, doubling a letter $\sigma_i$;
\item applying either the braid relation or the far commutativity one.
\end{enumerate}
In~\cite[Corollary~4.3]{GM}, the authors prove that given a positive braid diagram~$\V$, the Morton--Franks--Williams inequality is sharp for
the closure of $\V$ if and only if $\V$ is GMM. In particular, the closures of all GMM braid diagrams are both optimal and minimal (see~Proposition~\ref{SomeHomogeneousAreMinimal}).

In a forthcoming paper, we will extend the class of the closures of GMM braid diagrams to other link templates, and we will study 
the sharpness of the Morton--Franks--Williams inequality for the corresponding link diagrams.

\subsection{Discussion}

At the end of the~19th century, P.\ G.~Tait made three conjectures, called the {\it Tait~conjectures}, concerning alternating knots and links (see an expository note~\cite{Menasco19}):
\begin{enumerate}
\item any reduced alternating diagram is minimal;
\item any two reduced alternating diagrams that represent the same link have the same writhe;
\item any two reduced alternating diagrams that represent the same link are related through a finite sequence of diagram transformations referred to as {\it flypes}. 
\end{enumerate}
All of the Tait conjectures hold. 
In \cite[Main~Theorem]{FlypingConjecture}, W. Menasco and M. B. Thistlethwaite prove the third Tait conjecture, called the {\it Tait flyping conjecture}. 
The second Tait conjecture follows from the third one. 
The first Tait conjecture
was proved
by~L.~Kauffman, M.~B.~Thistlethwaite, and K.~Murasugi independently by using the Jones polynomial (see \cite[Theorem 2.10]{Kauffman}, \cite[Theorem~2]{Thistlethwaite87}, and \cite[Theorem A]{M}).
Namely, they showed that the number of crossings of any diagram of a link $\L$ is bounded from below by 
the breadth of
the Jones polynomial
of $\L$. Furthermore, for any reduced alternating link diagram, this inequality is sharp.

In \cite[Corollary 3.4]{Thistlethwaite88}, M.~B.~Thistlethwaite uses a generalization of the Jones polynomial referred to as the Kauffman polynomial and obtains similar minimality results 
for adequate link diagrams.
All reduced alternating diagrams are adequate. The simplest adequate non-alternating prime knots have~$10$ crossings, and there are $3$ such:~$10_{152}$, $10_{153}$, and $10_{154}$.

It is noteworthy that if a prime diagram is not alternating, the inequality concerning the breadth of the Jones polynomial is strict. This implies that any minimal diagram of each prime alternating link is alternating. 
In \cite[Theorem 10]{LT}, W. B. R. Lickorish and M. B. Thistlethwaite use this fact to show that some specific link diagrams referred to as reduced Montesinos ones are minimal, and thus they calculate the crossing number of all Montesinos links.

Traditionally, one indexes tables of (non-oriented) prime knots by the crossing number and one counts mirror images as a single knot type. 
It follows from the definitions that if a link diagram is locally twisted homogeneous, then its mirror image too.
A similar result holds for the reverse of a link diagram. 

Table \ref{Table} shows whether a prime knot up to $9$ crossings is adequate, Montesinos, or locally twisted homogeneous.
The majority of prime knots up to $9$ crossings are adequate. In~particular, there are only $3$ non-adequate knots $8_{19}$, $8_{20}$, and $8_{21}$ among the $36$ prime knots up to $8$ crossings, and there are only $8$ non-adequate knots $9_{42}, 9_{43}, \ldots, 9_{49}$ among the~$49$ prime knots with $9$ crossings. Moreover, all prime knots up to $9$ crossings except~$8_{20}$, $8_{21}$, $9_{42}$, $9_{44}$, $9_{45}$, $9_{46},$ and $9_{48}$ are homogeneous. For tables of Montesinos knots, we follow~\cite{Dun01}. We conclude that the positive knot $9_{49}$ is the only prime knot up to~$9$ crossings that does not admit a visually minimal diagram yet.

\begin{table}[H]
\centering
\begin{tabular}{|c|c|c||l|}
\cline{1-4}
A & M & LTH & \\ \cline{1-4}
$+$ & $+$ & $+$ & $3_1$, $4_1$, $5_1$, $6_2$, $6_3$, $7_1$, $7_4$, $7_6-7_7$, $8_2$, $8_5$, $8_7$, $8_9-8_{10}$, \\
    &     &     & $8_{12}$, $9_{1}$, $9_{11}$, $9_{17}$, $9_{20}$, $9_{22}$, $9_{24}$, $9_{26}-9_{28}$, $9_{30}-9_{31}$, $9_{36}$ \\ \cline{1-4}
$+$ & $+$ & $-$ & $5_2$, $6_1$, $7_2-7_3$, $7_5$, $8_1$, $8_3-8_4$, $8_6$, $8_8$, $8_{11}$, $8_{13}-8_{15}$,  \\
    &     &     & $9_{2}-9_{10}$, $9_{12}-9_{16}$, $9_{18}-9_{19}$, $9_{21}$, $9_{23}$, $9_{25}$, $9_{35}$, $9_{37}$ \\ \cline{1-4}
$+$ & $-$ & $+$ & $8_{16}-8_{18}$, $9_{29}$, $9_{32}-9_{34}$, $9_{40}$ \\ \cline{1-4}
$-$ & $+$ & $+$ & $8_{19}$, $9_{43}$ \\ \cline{1-4}
$-$ & $+$ & $-$ & $8_{20}-8_{21}$, $9_{42}$, $9_{44}-9_{46}$, $9_{48}$ \\ \cline{1-4}
$+$ & $-$ & $-$ & $9_{38}-9_{39}$, $9_{41}$ \\ \cline{1-4}
$-$ & $-$ & $+$ & $9_{47}$ \\ \cline{1-4}
$-$ & $-$ & $-$ & $9_{49}$ \\ \cline{1-4}
\end{tabular}
\PrintSkips{table}
\caption{A table of adequate (A), Montesinos (M), and locally twisted homogeneous (LTH) prime knots up to $9$ crossings.}\label{Table}
\end{table}

The closure of the third braid diagram shown in Figure~\ref{SolidBraidDiagrams} is locally twisted homogeneous and represents~$9_{47}$. 
We claim that the only known visual approach 
to establishing that 
at least one link diagram with $9$ crossings that represents $9_{47}$ is minimal
is to apply~Theorem~\ref{Theorem1}.
Equivalently (see \ref{Table}), we claim that~$9_{47}$ does not admit a link diagram that is the closure of a GMM braid diagram.
To prove this, we refer to the following facts.
First, any minimal diagram of a link that admits optimal positive one is optimal (see Proposition~\ref{SomeHomogeneousAreMinimal}). Second, any two optimal diagrams of the same link have the same writhe (see~\cite[Theorem~10]{DP13}). By combining these facts, we see that any minimal diagram of a link that admits optimal positive one is positive. However, the minimal diagram of~$9_{47}$ mentioned above is not positive. 

Recall that all adequate prime knots up to $9$ crossings are alternating.
All locally twisted alternating prime knots up to $9$ crossings are the closures of alternating braids. Besides, all locally twisted positive prime knots up to $9$ crossings are torus ones, and there are $5$ such:~$3_1$, $5_1$, $7_1$, $8_{19}$, and $9_1$.
The closure of the second braid diagram shown in~Figure~\ref{SolidBraidDiagrams} is locally twisted homogeneous and represents $9_{43}$. A simple enumeration shows that~$9_{43}$ and~$9_{47}$ are the only two non-torus non-alternating locally twisted homogeneous prime knots up to $9$ crossings. 

In contrast to alternating diagrams, there are non-minimal reduced homogeneous ones (see~\cite[Figure~12]{Cromwell89}). Also, there exist positive links that admit both non-minimal positive diagrams and non-homogeneous minimal ones 
(see~\cite[Theorem~1]{Stoimenow}). 
Thus, the problem of visual determining the crossing number of a link given by a homogeneous diagram remains unsolved.
Besides, no analog of the Tait flyping conjecture has been found for homogeneous links so far. 

A natural question is whether one can read off topological 
properties of the link given by a diagram~$\D$ 
(such as knottedness, splitness, primeness, and hyperbolicity) 
from~$\D$. 
In some sense, homogeneous links are visually knotted. Namely, a homogeneous link~$\L$ is trivial if and only if for some (and hence any) homogeneous diagram~$\D$ of~$\L$,
the Seifert graph of $\D$
has no cycles (see \cite[Theorem~3]{Cromwell89} and \cite[Corollary 7.6.3]{CromwellBook}). 
A link $\L$ is said to be {\it split} if~$\L$ admits a disconnected diagram. 
Homogeneous links are visually split in the sense that a homogeneous link~$\L$ is split if and only if
some (and hence any) homogeneous diagram
of~$\L$
is disconnected (see \cite[Corollary~3.1]{Cromwell89} and~\cite[Corollary~7.6.4]{CromwellBook}). 
In~\cite{Cromwell93}, the author conjectures that any link given by a prime homogeneous diagram is prime. 
Equivalently, the conjecture states that 
homogeneous links are visually prime in the sense that a homogeneous link $\L$ is prime if and only if some (and hence any) homogeneous diagram of~$\L$ is prime. This conjecture holds for both alternating (see~\cite[Theorem~1]{Menasco84} and \cite[Theorem~4.4]{Lickorish}) and positive links (see~\cite[Theorem~1.4]{Ozawa91} and~\cite[Theorem~1.2]{Cromwell93}). 
Finally, in \cite[Corollary~2]{Menasco84}, the author proves that if a non-torus alternating link~$\L$ is both prime and non-split, then~$\L$ is hyperbolic.  
However, the problem of visual determining whether a link given by a homogeneous diagram is hyperbolic remains unsolved.

The paper is organized as follows. 
In Section~2, we give a detailed definition of both locally twisted diagrams and locally twisted homogeneous ones.
In Section~3, we state the Morton--Franks--Williams inequality explicitly,
reduce Theorem \ref{Theorem1} to Corollary \ref{Corollary1}, 
and 
give a definition of 
resolution trees for the skein polynomial. 
The remaining part of the paper aims to prove Theorem~\ref{Theorem2}. The basis of our proof is the techniques developed in~\cite{BIRAL}.
In~Section~4, we introduce resolution trees that we call coherent. 
In Section~5, we introduce castle structures for link diagrams to describe specific resolution trees that we call special coherent. 
In Section~6, we apply these resolution trees to prove~Theorem~\ref{Theorem2}.

\section{Locally twisted homogeneous diagrams}

This section aims to define locally twisted homogeneous link diagrams. 
First, we introduce locally twisted braid words. Then, we define link templates, which generalize the Alexander closure of braids construction, and we give a complete definition of locally twisted link diagrams. Finally, we define homogeneous link diagrams introduced in \cite{Cromwell89}.

\subsection{Locally twisted braid words}

Given a set of symbols~$S$, a finite sequence of elements of $S$ is called a {\it word in an alphabet $S$}. We write words without commas and we denote by $S^\ast$ the set of all words in the alphabet $S$. In this paper, we consider alphabets of the form $S \subseteq \{\sigma_1,\sigma_1^{-1}, \sigma_2, \sigma_2^{-1}, \ldots\}$.

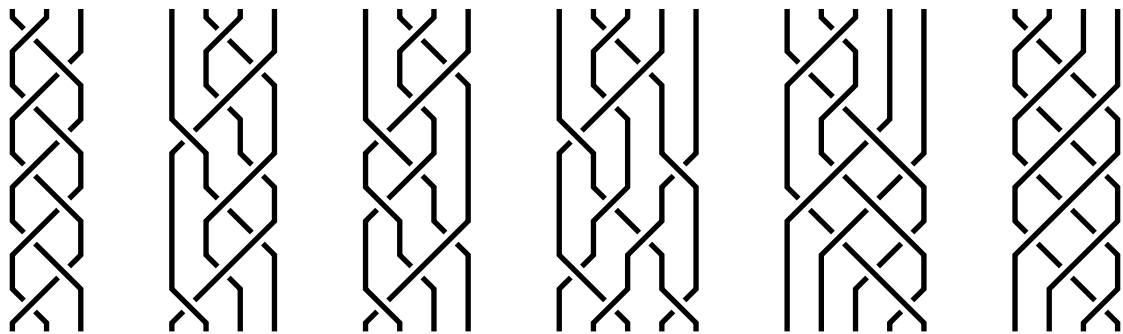
\begin{figure}[H]
\centering
\begin{tikzpicture}[scale=0.45, every node/.style={scale=0.45}]
\pic[
  line width=2pt,
  braid/control factor=0,
  braid/nudge factor=0,
  braid/gap=0.11,
  braid/number of strands = 3,
  name prefix=braid,
] at (0,0) {braid={
s_1^{-1}s_2s_1^{-1}s_2s_1^{-1}s_2s_1^{-1}s_2s_1^{-1}
}};
\end{tikzpicture}\hspace{1cm}
\begin{tikzpicture}[scale=0.45, every node/.style={scale=0.45}]
\pic[
  line width=2pt,
  braid/control factor=0,
  braid/nudge factor=0,
  braid/gap=0.11,
  braid/number of strands = 4,
  name prefix=braid,
] at (0,0) {braid={
s_2^{-1}s_3^{-1}s_2^{-1}s_1s_3^{-1}s_2^{-1}s_3^{-1}s_2^{-1}s_1
}};
\end{tikzpicture}\hspace{1cm}
\begin{tikzpicture}[scale=0.45, every node/.style={scale=0.45}]
\pic[
  line width=2pt,
  braid/control factor=0,
  braid/nudge factor=0,
  braid/gap=0.11,
  braid/number of strands = 4,
  name prefix=braid,
] at (0,0) {braid={
s_2^{-1}s_3^{-1}s_2^{-1}s_1s_2^{-1}s_1s_3^{-1}s_2^{-1}s_1
}};
\end{tikzpicture}\hspace{1cm}
\begin{tikzpicture}[scale=0.45, every node/.style={scale=0.45}]
\pic[
  line width=2pt,
  braid/control factor=0,
  braid/nudge factor=0,
  braid/gap=0.11,
  braid/number of strands = 5,
  name prefix=braid,
] at (0,0) {braid={
s_2^{-1}s_3^{-1}s_2^{-1}s_1s_4s_2^{-1}s_3^{-1}s_1s_2^{-1}-s_4
}};
\end{tikzpicture}\hspace{1cm}
\begin{tikzpicture}[scale=0.45, every node/.style={scale=0.45}]
\pic[
  line width=2pt,
  braid/control factor=0,
  braid/nudge factor=0,
  braid/gap=0.11,
  braid/number of strands = 5,
  name prefix=braid,
] at (0,0) {braid={
s_2^{-1}s_1^{-1}s_2^{-1}s_3s_2^{-1}-s_4s_3-s_1^{-1}s_2^{-1}-s_4s_3s_4
}};
\end{tikzpicture}\hspace{1cm}
\begin{tikzpicture}[scale=0.45, every node/.style={scale=0.45}]
\pic[
  line width=2pt,
  braid/control factor=0,
  braid/nudge factor=0,
  braid/gap=0.11,
  braid/number of strands = 4,
  name prefix=braid,
] at (0,0) {braid={
s_1^{-1}s_2^{-1}s_1^{-1}-s_3^{-1}s_2^{-1}s_1^{-1}-s_3^{-1}s_2^{-1}s_1^{-1}-s_3^{-1}s_2^{-1}s_3^{-1}
}};
\end{tikzpicture}
\caption{
Examples of braid diagrams.
}
\label{SolidBraidDiagrams}
\end{figure}

Given~$n \geq 2,$ an element of~$\{\sigma_1, \sigma_1^{-1}, \ldots, \sigma_{n-1}, \sigma_{n-1}^{-1}\}^\ast$ 
is called a {\it braid word}.
We visualize braid words by their $n$ strand diagrams, which we draw vertically from top to bottom. For example, see the second picture in~Figure~\ref{SolidBraidDiagrams} for the braid diagram with $4$ strands corresponding to~$\sigma_2\sigma_3\sigma_2\sigma_1^{-1}\sigma_3\sigma_2\sigma_3\sigma_2\sigma_1^{-1}$. 

We denote by $\mathcal{B}_n$ the braid group with $n$ strands. Recall that $\B_n$ admits the following presentation with the standard Artin generators:
\begin{align*}
\B_n \simeq \langle \sigma_1, \ldots, \sigma_{n-1} \mid \sigma_k\sigma_{k+1}\sigma_k=\sigma_{k+1}\sigma_k\sigma_{k+1}, \ 1 \leq k \leq n-1; \ \sigma_i\sigma_j=\sigma_j\sigma_i, \ |i-j|\geq 2 \rangle.
\end{align*}
The relation~$\sigma_k\sigma_{k+1}\sigma_k=\sigma_{k+1}\sigma_k\sigma_{k+1}$ is called the {\it braid relation}, and the relation $\sigma_i\sigma_j=\sigma_j\sigma_i$ is called the {\it far commutativity} relation. An element of $\B_n$ is called a {\it braid} with $n$ strands. 

Given $i,j \in \{1,2,\ldots,n\}$ such that $i<j$, let
\begin{align}\label{DeltaWord1}
\delta_{i,j} := (\sigma_i \sigma_{i+1} \ldots \sigma_{j-1} \sigma_{j})(\sigma_i \sigma_{i+1} \ldots \sigma_{j-1}) \ldots (\sigma_i \sigma_{i+1}) \sigma_i,
\end{align}
and let $\Delta_{i,j} \in \B_n$ be the braid corresponding to $\delta_{i,j}$. The braid $\Delta_{1,n}$ is referred to as the {\it half twist}, the {\it fundamental braid}, and the {\it Garside element} in~$\B_n$. The braid~$\Delta_{1,n}^2 \in \B_n$ is referred to as the {\it full twist}.

In \cite{Stallings}, J. Stallings introduces the concept of homogeneous braid words. Namely, given $n \geq 2$ and given~$r = (r_1, \ldots, r_{n-1}) \in \{1,-1\}^{n-1}$, a braid word $w \in \{\sigma_1^{\pm 1}, \ldots, \sigma_{n-1}^{\pm 1}\}^\ast$ is said to be {\it $r$-homogeneous} if~$w$ contains none of the letters $\sigma_1^{-r_1}, \sigma_2^{-r_2}, \ldots, \sigma_{n-1}^{-r_{n-1}}$.

Let $r \in \{1,-1\}^{n-1}$. Let $i_1,i_2,\ldots,i_m \in \{1,2,\ldots,n\}$ be such that $$1~=~i_1~<~i_2~<~\ldots~<~i_m~=~n$$ and for all $k \in \{1,2,\ldots,m-1\}$, one has $r_{i_k}=r_{i_k+1}=\ldots=r_{i_{k+1}-1}$ and $r_{i_{k+1}-1} \neq r_{i_{k+1}}$. 

Given an $r$-homogeneous braid word $w$ and $k \in \{1,2,\ldots,m-1\}$, we say that the braid word obtained from $w$ by deleting all letters except~$\sigma_{i_k}^{\pm 1}, \sigma_{i_k+1}^{\pm 1}, \ldots, \sigma_{i_{k+1}-1}^{\pm 1}$ is the {\it uniform layer of index $k$ of $w$}.

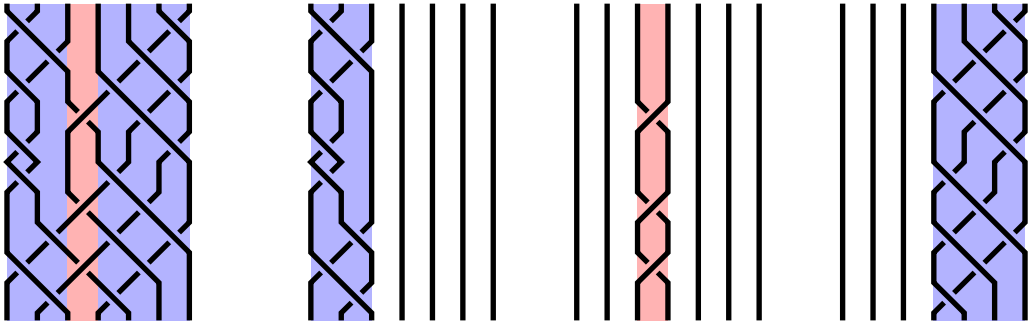
\begin{figure}[H]
\centering
\begin{tikzpicture}[scale=0.4, every node/.style={scale=0.4}]
\fill[blue!30!white] (0,0) rectangle ++(2,-10.5);
\fill[red!30!white] (2,0) rectangle ++(1,-10.5);
\fill[blue!30!white] (3,0) rectangle ++(3,-10.5);
\pic[
  line width=2pt,
  braid/control factor=0,
  braid/nudge factor=0,
  braid/gap=0.11,
  braid/number of strands = 7,
  name prefix=braid,
] at (0,0) {braid={
s_1-s_6
s_2-s_5
s_1-s_4-s_6
s_3^{-1}-s_5
s_1-s_6
s_1-s_4
s_3^{-1}-s_5
s_2-s_4-s_6
s_1-s_3^{-1}-s_5
s_2-s_4
}};
\end{tikzpicture}\hspace{1.4cm}
\begin{tikzpicture}[scale=0.4, every node/.style={scale=0.4}]
\fill[blue!30!white] (0,0) rectangle ++(2,-10.5);
\pic[
  line width=2pt,
  braid/control factor=0,
  braid/nudge factor=0,
  braid/gap=0.11,
  braid/number of strands = 7,
  name prefix=braid,
] at (0,0) {braid={
s_1s_2s_11s_1s_11s_2s_1s_2
}};
\end{tikzpicture}\hspace{0.9cm}
\begin{tikzpicture}[scale=0.4, every node/.style={scale=0.4}]
\fill[red!30!white] (2,0) rectangle ++(1,-10.5);
\pic[
  line width=2pt,
  braid/control factor=0,
  braid/nudge factor=0,
  braid/gap=0.11,
  braid/number of strands = 7,
  name prefix=braid,
] at (0,0) {braid={
111s_3^{-1}11s_3^{-1}1s_3^{-1}1
}};
\end{tikzpicture}\hspace{0.9cm}
\begin{tikzpicture}[scale=0.4, every node/.style={scale=0.4}]
\fill[blue!30!white] (3,0) rectangle ++(3,-10.5);
\pic[
  line width=2pt,
  braid/control factor=0,
  braid/nudge factor=0,
  braid/gap=0.11,
  braid/number of strands = 7,
  name prefix=braid,
] at (0,0) {braid={
s_6
s_5
s_4-s_6
s_5
s_6
s_4
s_5
s_4-s_6
s_5
s_4
}};
\end{tikzpicture}
\caption{
A braid diagram and the corresponding uniform layers.
}
\label{LocallyTwistedBraidWord}
\end{figure}

Let $\mathcal{H}_r^+$ (resp.\ $\mathcal{H}_r^{-}$) be the class of all $r$-homogeneous braid words 
$w$
such that for each~$k~\in~\{1,2,\ldots,m-1\}$ with $r_{i_k}~=~1$ (resp.\ $r_{i_k}~=~-1$), 
the index~$k$ uniform layer of~$w$ 
admits a decomposition of the form $v_1v_2v_3$ such that both $v_1$ and $v_3$ represent $\Delta_{i_k,i_{k+1}}$ (resp.\ $\Delta_{i_k,i_{k+1}}^{-1}$). 
The uniform layers of $w \in \mathcal{H}_r^{+} \cap \mathcal{H}_r^{-}$ are called {\it braid words with the full twist}.

\begin{definition}
Given $\# \in \{+,-\}$ and $n\geq 2$, we say that a braid word $w \in \{\sigma_1^{\pm 1}, \ldots, \sigma_{n-1}^{\pm 1}\}^\ast$ is {\it locally~$\#$twisted} (resp.\ {\it locally twisted}), if there exist $r \in \{1,-1\}^{n-1}$ such that $w \in \mathcal{H}_r^{\#}$ (resp.\ ~$w \in \mathcal{H}_r^{+} \cap \mathcal{H}_r^{-}$).
\end{definition}

For example, the braid word~$w$ corresponding to the braid diagram shown in Figure~\ref{LocallyTwistedBraidWord} (on the left) is locally twisted with~$r = (-1,-1,1,-1,-1,-1)$. Besides, all braid diagrams shown in Figure~\ref{SolidBraidDiagrams} correspond to locally twisted braid words.

\subsection{Link templates}

Any braid diagram gives rise to a link diagram via the {\it Alexander closure} (see Figure~\ref{Closure}). To define locally twisted link diagrams, we generalize this construction by modifying that of \cite{Yamada}.

Let $\C = \{C_1, C_2, \ldots, C_s\}$ 
and
$\A = \{\a_1,\a_2,\ldots,\a_t\}$
be a set of disjoint oriented circles 
and 
a set of disjoint oriented simple closed arcs, respectively, on the plane. 
The ordered pair 
$(\C,\A)$
is called a {\it template}
if $\partial \left(\bigcup \A\right) \subseteq \left(\bigcup \C\right)$ and
for all 
$x \in \left(\bigcup \C \right)\cap \left(\bigcup \A\right)$, 
there is a neighborhood of $x$ diffeomorphic to one of the pictures shown in Figure \ref{LocalIntersection} (on the left).

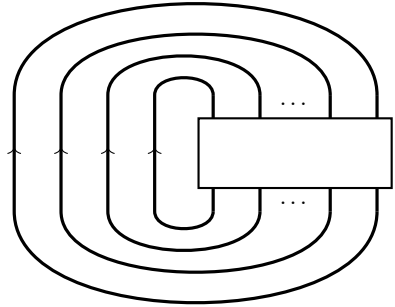
\begin{figure}
\centering
\begin{tikzpicture}[scale=0.77, every node/.style={scale=0.82}]
\draw[very thick] (-1.4,0) -- (-1.4,2);
\draw[very thick] (-.6,0) -- (-.6,2);
\draw[very thick] (.6,0) -- (.6,2);
\draw[very thick] (1.4,0) -- (1.4,2);
\draw[very thick ] (-1.4,2) .. controls ++(0,.4) and ++(0,.4) .. (-2.4,2)
    -- (-2.4,0) .. controls ++(0,-.4) and ++(0,-.4) .. (-1.4,0);
\draw[very thick ] (-.6,2) .. controls ++(0,.9) and ++(0,.9) .. (-3.2,2)
               -- (-3.2,0) .. controls ++(0,-.9) and ++(0,-.9) .. (-.6,0);
\draw[very thick ] (.6,2) .. controls ++(0,1.4) and ++(0,1.4) .. (-4,2)
                -- (-4,0) .. controls ++(0,-1.4) and ++(0,-1.4) .. (.6,0);
\draw[very thick ] (1.4,2) .. controls ++(0,2.1) and ++(0,2.1) .. (-4.8,2)
                -- (-4.8,0) .. controls ++(0,-2.1) and ++(0,-2.1) ..
(1.4,0);
\draw[thick, fill=white, draw=black] (-1.65,.4) -- (-1.65,1.6) --
(1.65,1.6) -- (1.65,.4) -- cycle;
\node at (0,.15) {$
\dots$};
\node at (0,1.85) {$
\dots$};
\draw [->] (-3.2,1) -- (-3.2,1.1); 
\draw [->] (-4,1) -- (-4,1.1); 
\draw [->] (-4.8,1) -- (-4.8,1.1); 
\draw [->] (-2.4,1) -- (-2.4,1.1); 
\end{tikzpicture}
\caption{The Alexander closure of a braid diagram.}
\label{Closure}
\end{figure}

\begin{figure}
\centering
\includegraphics[width = 14cm]{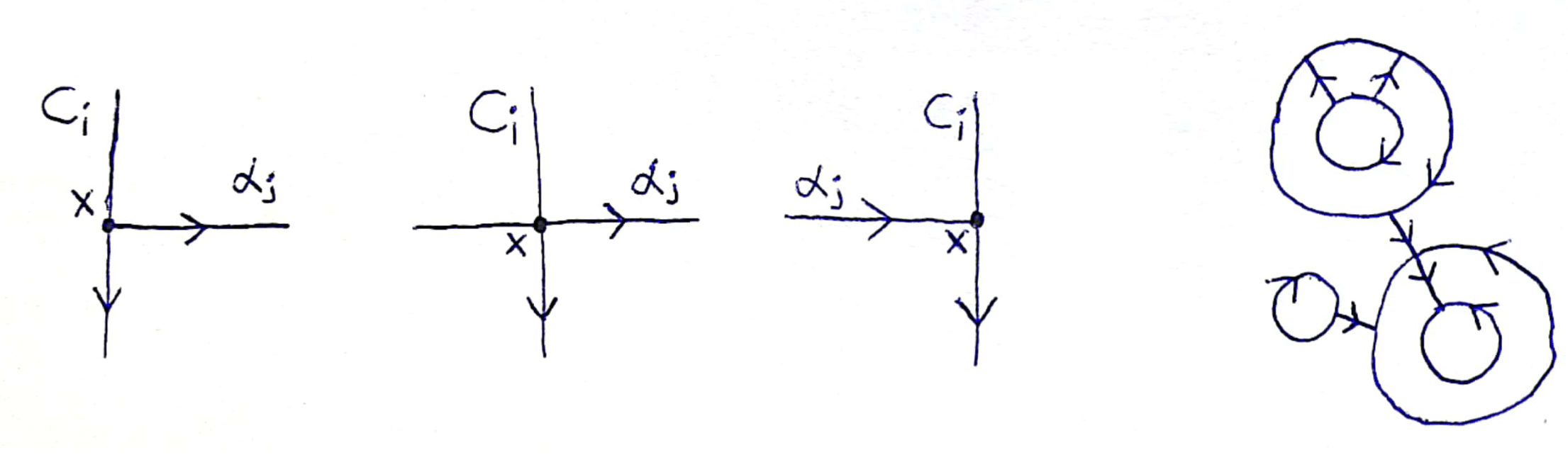}
\caption{Any point 
$x \in \left(\bigcup \C \right)\cap \left(\bigcup \A\right)$
has a neighborhood that looks like one of these (on the left). An example of a template (on the right).}
\label{LocalIntersection}
\end{figure}

\begin{figure}
\centering
\includegraphics[width = 14cm]{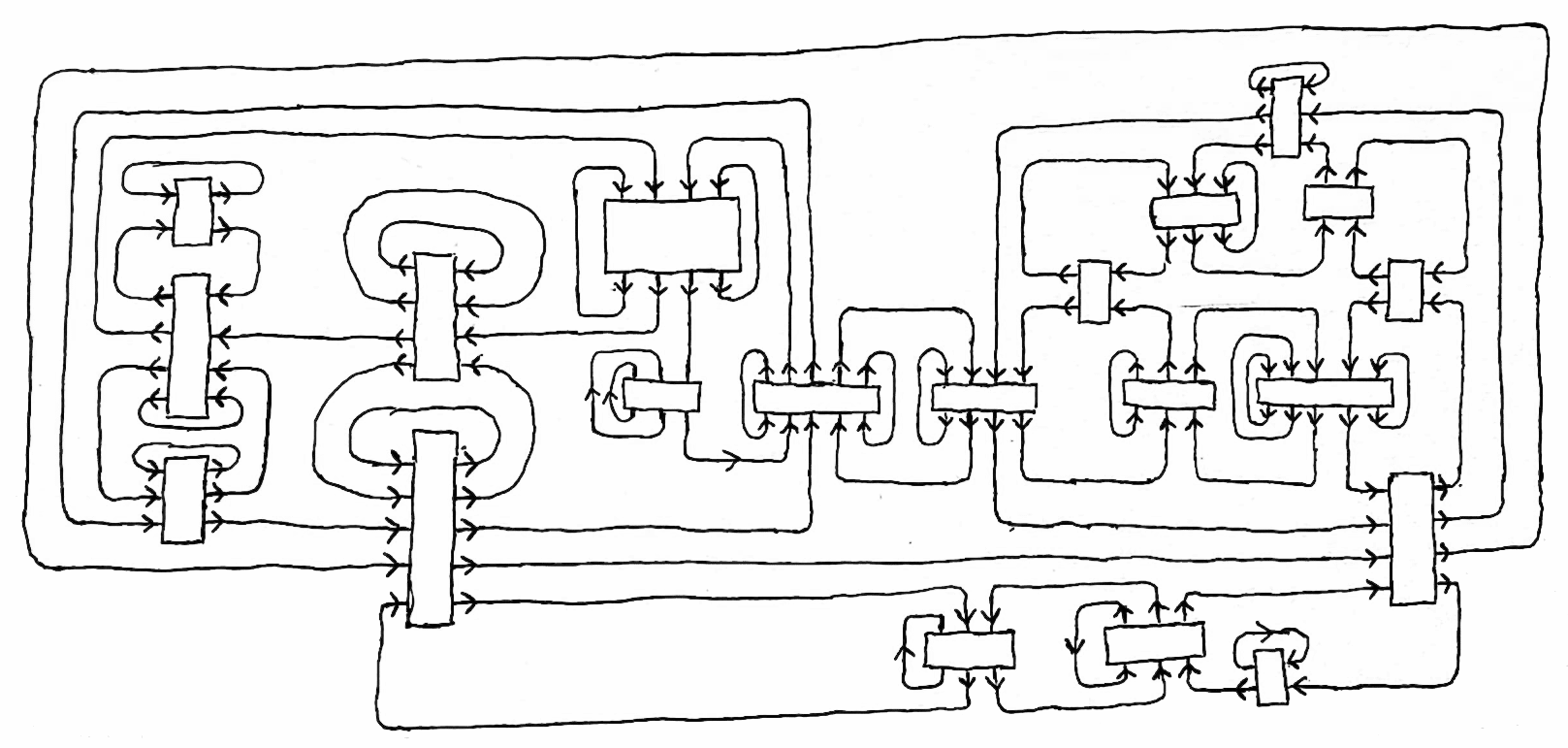}
\caption{Construction of link diagrams determined by a template.}
\label{WSSC}
\end{figure}

Let $(\C,\A)$ be a template. Given $\a \in \A$, denote by $\left\lVert \a \right\rVert$ the number of circles $C \in \C$ intersecting~$\a$.
Note that for all $\a \in \A$, one has $\left\lVert \a \right\rVert \geq 2$, and the arc $\a$ intersects each circle~$C \in \C$ no more than once. A function 
$$\p \colon \A \longrightarrow \{\sigma_i^{\pm 1} \mid i \in \{1,2,3, \ldots\}\}^\ast$$ is called a {\it braid placement} if for all $\a \in \A$, one has 
$$\p(\a) \in \{\sigma_i^{\pm 1} \mid i \in \{1,2,\ldots,\left\lVert \a \right\rVert-1\}\}^\ast.$$

Any pair consisting of a template $(\C,\A)$ and a braid placement $\p$ determines a link diagram as follows. For each $\a \in \A$, replace $\a$ in $\left(\bigcup \C \right)\cup \left(\bigcup \A\right) \subseteq \mathbb{R}^2$ by a rectangle (see Figure \ref{WSSC}) and insert a braid diagram corresponding to~$\p(\a)$ according to the orientation. We say that the resulting link diagram is {\it determined} by $\p$. Note that any link diagram arises in this way.

We say that a template $(\C,\A)$ is {\it knitted} if for any $C_1, C_2 \in \C$ there exists at most one arc~$\a~\in~\A$ intersecting both $C_1$ and $C_2$. For example, for each $n\geq 2$, the template corresponding to the Alexander closure of a braid with $n$ strands is knitted. Besides, the template corresponding to the picture shown in Figure \ref{WSSC} is knitted.

\begin{definition}
Let $\# \in \{+,-\}$. Let $(\C, \A)$ be a 
knitted template, and let $\p$ be a braid placement. 
The link diagram determined by $\p$ is said to be {\it locally $\#$twisted} (resp.\ {\it locally twisted}) if for all $\a \in \A$, 
the braid word $\p(\a)$ is locally $\#$twisted (resp.\ locally twisted).
\end{definition}

\subsection{Homogeneous link diagrams}

Recall that we assume all link diagrams to be oriented.
A~crossing is said to be positive (resp.\  negative) if it has the form on the left (resp.\ right) picture in~Figure~\ref{Smoothing}.

\begin{figure}[H]
\centering
\includegraphics[width = 11cm]{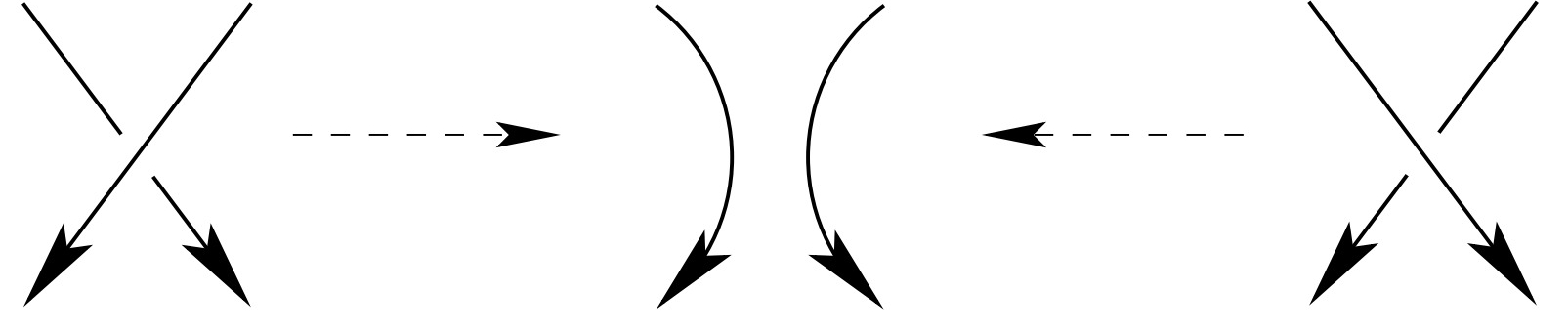}
\caption{Smoothing of a crossing.}
\label{Smoothing}
\end{figure}

Let $\D$ be a link diagram. Simple closed curves obtained by smoothing all crossings of~$\D$ are called {\it Seifert circles} of $\D$.

Let $\Gamma(\D)$ be the undirected multigraph whose vertex set is the set of all Seifert circles of $\D$ and whose 
edge set is the set of all crossings of $\D$. Each edge of $\Gamma(\D)$ endowed with a sign according to the type of the corresponding crossing.
The graph $\Gamma(\D)$ is referred to as the {\it Seifert graph} of~$\D$. Note that $\Gamma(\D)$ is bipartite, that is, it contains no odd-length cycles.

To define homogeneous link diagrams, we follow \cite{Cromwell89}.
Let $\Gamma$ be a Seifert graph.
A~vertex $v$ of $\Gamma$ is said to be a {\it cut vertex} if the number of components of~$\Gamma-v$ is greater than that of $\Gamma$. Suppose that $\Gamma$ contains a cut vertex $v$, and let $\Gamma_1, \Gamma_2, \ldots, \Gamma_n$ be the components of $\Gamma-v$. We say that the $n$ subgraphs $\Gamma_1 \cup v, \Gamma_2 \cup v, \ldots, \Gamma_n \cup v$ are obtained from $\Gamma$ by {\it cutting} $\Gamma$ at $v$. Cutting $\Gamma$ at each of its cut vertices produces a set of components, each one being a subgraph of~$\Gamma$ containing no cut vertices. Such a component is called a {\it block} of $\Gamma$. A block $H$ of~$\Gamma$ is said to be {\it homogeneous} if all edges of~$H$ have the same sign.

A diagram~$\D$ is said to be {\it homogeneous} if each block of~$\Gamma(\D)$ is homogeneous. 
A link is said to be {\it homogeneous} 
if it admits a homogeneous 
diagram.

We list several diagrammatic properties 
related to 
homogeneity.
First, if a diagram~$\D$ is positive, that is, all crossings of~$\D$ are positive, then~$\D$ is homogeneous. Second, if~$\D$ is alternating, then $\D$ is homogeneous. Third, the Seifert circles of a link diagram $\D$ divide the plane into regions. Suppose that each of these regions contains crossings of the same type. Then $\D$ is homogeneous. Fourth, let $(\C, \A)$ be a template, and let $\p$ be a braid placement. If the link diagram determined by $\p$ is homogeneous, then for each~$\a~\in~\A$, the braid word $\p(\a)$ is $r$-homogeneous for some $r \in \{1,-1\}^{\left\lVert \a \right\rVert-1}$. In particular, if the link diagram determined by $\p$ is alternating, then for each~$\a~\in~\A$, the braid word $\p(\a)$ is either $r$-homogeneous or $(-r)$-homogeneous for $r = (-1,1,-1,1, \ldots, (-1)^{\left\lVert \a \right\rVert}) \in \{1,-1\}^{\left\lVert \a \right\rVert-1}$ whenever for all $i \in \{1,2,\ldots,\left\lVert \a \right\rVert-1\}$, the word $\p(\a)$ contains either $\sigma_i$ or $\sigma_i^{-1}$.

\begin{definition}
A link diagram $\D$ is said to be {\it locally twisted homogeneous} if $\D$ is both locally twisted and homogeneous.
\end{definition}

\section{The skein polynomial}

This section aims to state the Morton--Franks--Williams inequality, reduce Theorem \ref{Theorem1} to Corollary \ref{Corollary1}, and introduce resolution trees for the skein polynomial. 

\subsection{The Morton--Franks--Williams inequality}

In \cite{Homfly, PT}, the authors prove that there is a unique function that maps each link diagram~$\D$ to a two-variable Laurent polynomial~$\P(\D;a,z) \in \mathbb{Z}[a^{\pm 1},z^{\pm 1}]$ such that:
\begin{enumerate}
\item if two link diagrams $\D$ and $\D^\prime$ represent the same link, then~$\P(\D; a,z) = \P(\D^\prime; a,z)$;
\item one has
\begin{align*}
a \P(\D_+; a,z) - a^{-1} \P(\D_-; a,z) = z\P(\D_0, a,z)
\end{align*}
whenever $\D_+$, $\D_0$, and $\D_-$ are link diagrams that coincide except at a small region where the diagrams are presented as in Figure~\ref{Smoothing}, respectively;
\item if $\D$ is a knot diagram 
with zero crossings,
then~$\P(\D; a,z) = 1$.
\end{enumerate}
The polynomial~$\P(\D;a,z)$ is referred to as the {\it skein polynomial}, the {\it HOMFLY polynomial}, the {\it HOMFLY-PT polynomial}, the {\it generalized Jones polynomial}, and the {\it twisted Alexander polynomial} of $\D$. The second condition is referred to as the {\it skein relation}.

Given a link $\L$, let~$\P(\L;a,z) := \P(\D;a,z)$ for some (and hence any) diagram $\D$ of $\L$. Denote by $s(\D)$ the number of Seifert circles of~$\D$ and by 
$\omega(\D)$ the {\it writhe} of $\D$, that is, the difference between the number of positive and the number of negative crossings of $\D$.

Let $\Q_i(\L;z) \in \mathbb{Z}[z^{\pm 1}]$ be polynomials such that~$\P(\L;a,z) = \sum_{i=e}^E \Q_i(\L;z) a^i$, $\Q_e(\L;z)~\neq~0$, and $\Q_E(\L;z)~\neq~0$. In \cite[Theorem 1]{Morton}, the author proves that for any diagram~$\D$ of~$\L$, one has
\begin{align}\label{DegSeif0}
-\omega(\D) - (s(\D)-1) \leq e \leq E \leq -\omega(\D) + (s(\D)-1).
\end{align}
In particular,
\begin{align}\label{MFW}
(E-e)/2 + 1 \leq s(\D).
\end{align}
The latter is referred to as the {\it Morton--Franks--Williams inequality}. 
Given a link diagram~$\D$, we say that \eqref{MFW} is {\it sharp} for $\D$ if $(E-e)/2 + 1 = s(\D)$. In this case, $\D$ is optimal. 
Note~that~\eqref{MFW} is sharp if and only if both outer inequalities in \eqref{DegSeif0} are sharp.

\subsection{Reduction of Theorem \ref{Theorem1} to Corollary \ref{Corollary1}}

Given a link diagram $\D$, denote by $|\D|$ the number of crossings of~$\D$.

Let $\P_i(\L; a) \in \mathbb{Z}[a^{\pm 1}]$ be polynomials such that~$\P(\L;a,z) = \sum_{i=m}^M \P_i(\L;a) z^i$, $\P_m(\L;a)~\neq~0$, and $\P_M(\L;a)~\neq~0$. In \cite[Theorem 2]{Morton}, the author proves that for any diagram~$\D$ of~$\L$, one has
\begin{align}\label{DegCrSeif}
M \leq |\D| - s(\D) + 1.
\end{align}
By combining \eqref{MFW} and \eqref{DegCrSeif}, one has 
\begin{align*}
M+(E-e)/2 \leq |\D|.
\end{align*}

It turns out that for all homogeneous diagrams, \eqref{DegCrSeif} is sharp (see \cite[Theorem~4]{Cromwell89}, \cite[Theorem 7.6.2]{CromwellBook}, and \cite[Theorem~6]{HomogeneousSeifert}). It is worth noting that the sharpness of~\eqref{DegCrSeif} does not characterize homogeneous links since \eqref{DegCrSeif} is sharp for an almost positive diagram of $12n_{149}$, which is not homogeneous.

By using the sharpness mentioned above, we prove Proposition \ref{SomeHomogeneousAreMinimal} below.
Assertion {\rm(i)} was firstly observed in \cite[Proposition 7.4]{MurasugiA}.
This assertion implies that Theorem \ref{Theorem1} is a special case of Corollary \ref{Corollary1}.
For the sake of completeness, we provide a proof.

\begin{Proposition}\label{SomeHomogeneousAreMinimal}
Suppose a homogeneous diagram $\D$ of a link $\L$ is optimal.
Then
\begin{enumerate}
\item[{\rm(i)}] the diagram $\D$ is minimal;
\item[{\rm(ii)}] any minimal diagram of $\L$ is optimal.
\end{enumerate}
\end{Proposition}
\begin{proof}
Let us prove assertion {\rm(i)}. Let $\D^{\prime}$ be an arbitrary link diagram of~$\L$. 
Since~$\D$ is optimal, the inequality $s(\D) \leq s(\D^{\prime})$ holds. One has
\begin{align*}
|\D| \stackrel{\eqref{DegCrSeif}}{=} M + s(\D) - 1 \stackrel{\eqref{DegCrSeif}}{\leq} (|\D^{\prime}| - s(\D^{\prime}) + 1) + s(\D) - 1 \leq |\D^{\prime}|.
\end{align*}
Therefore, $\D$ is minimal.

Let us prove assertion {\rm(ii)}. Let $\D^\prime$ be a minimal diagram of~$\L$. In particular, $|\D^\prime| = |\D|$. One has
\begin{align*}
s(\D^\prime) \stackrel{\eqref{DegCrSeif}}{\leq} |\D^\prime| - M + 1 \stackrel{\eqref{DegCrSeif}}{=} |\D^\prime| - (|\D| - s(\D) + 1) + 1 = s(\D).
\end{align*}
Therefore, since 
$\D$ is optimal, $\D^\prime$ is optimal too.
\end{proof}

\subsection{Resolution trees}

The process of evaluating the skein polynomial of a link diagram by repeated application of the skein relation is referred to as {\it resolution}. One records this process schematically in a binary structure,
referred to as 
a {\it resolution tree} (see \cite[p.\ ~18]{GM}, \cite[p.\ 538]{Cromwell89}, and~\cite[p.\ 15]{Nakamura}),
a {\it resolving tree} (see \cite[p.\ 193]{LDH}, \cite[p.\ ~4]{BIRAL}, \cite[p.\ 183]{CromwellBook}, \cite[p.\ 4]{Homfly}, and \cite[p.\ 167]{Ad94}), and
a {\it computation tree} (see \cite[p.\ 99]{FW} and \cite[p.\ 653]{Kalman}).
Namely, given a link diagram $\D$, 
a resolution tree for $\D$ is
a weighted 
binary tree~$\T$ such that:
\begin{enumerate}
\item each node of $\T$ is a link diagram;
\item the root node of $\T$ is $\D$;
\item each leaf node of $\T$ represents an unlink;
\item each internal node has exactly two children. The corresponding three link diagrams are identical except at one crossing, and they relate by one of the two relations at that crossing, as shown in Figure \ref{Resolving}. 
\end{enumerate}

\begin{figure}[H]
\centering
\includegraphics[width = 12.5cm]{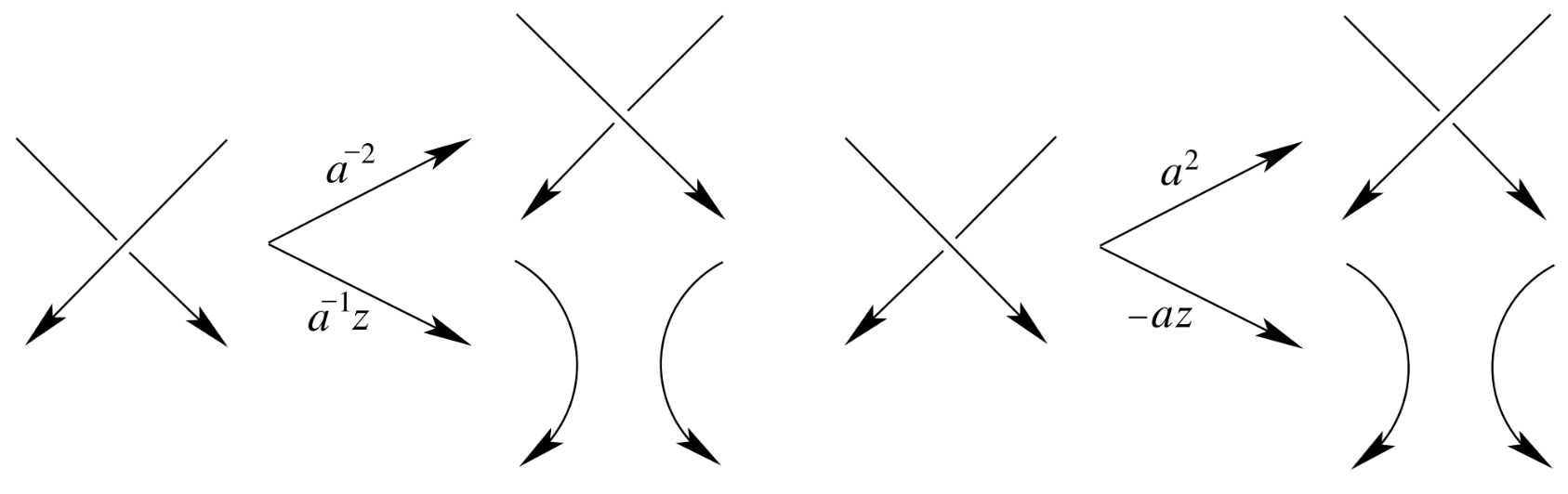}
\caption{The weight assignment for a resolution tree. The picture was taken
from \cite{LDH}.}
\label{Resolving}
\end{figure}

We refer to the lines connecting nodes of a binary tree as {\it branches}.

Any resolution tree $\T$ of $\D$ gives rise to a decomposition of~$\P(\D; a,z)$ as a sum, which is indexed by leaf nodes of $\T$, as follows. Calculations show that if $\D$ is the link diagram with zero crossings and with precisely $n$ link components, one has $$\P(\D; a,z) = ((a-a^{-1})z^{-1})^{n-1}.$$ Note that the skein relation admits the following equivalent forms:
\begin{align*}
\P(\D_+; a,z) &= a^{-2} \P(\D_-; a,z) + a^{-1}z\P(\D_0, a,z), \\
\P(\D_-; a,z) &= a^{2} \P(\D_+; a,z) - a z\P(\D_0, a,z).
\end{align*}
Let $\U$ be a leaf node of $\T$. Denote by $\gamma(\U)$ the number of link components of $\U$. Let~$P$ be a unique path on $\T$ from the root node $\mathcal{D}$ to the leaf node $\mathcal{U}$. It is easy to see that the contribution of $\U$ to~$\P(\D; a,z)$ is $((a-a^{-1})z^{-1})^{\gamma(\U)-1}$ multiplied by the weights of the branches in $P$. Let $t(\U)$ be the number of crossings of $\D$ that one smoothed in obtaining~$\U$, and $t^-(\U)$ be the number of negative crossings among the smoothed ones. As~Figure \ref{Resolving} shows, the degree of $a$ in the weight of a branch is equal to the change of writhe from 
the parent to the child.
Besides, a $z$ term in the weight of the branch indicates that the 
child is obtained from the 
parent 
by a crossing smoothing and a negative sign in the weight indicates that the smoothed crossing is negative. It follows that 
\begin{align*}
\P(\D; a,z) = \sum\limits_{\U \in \T_\circ} (-1)^{t^-(\U)} z^{t(\U)} a^{\omega(\U) - \omega(\D)} ((a-a^{-1})z^{-1})^{\gamma(\U)-1},
\end{align*}
where $\mathcal{T}_\circ$ is the set of leaf nodes of $\mathcal{T}$. 

Following \cite{BIRAL},
we think of a resolution tree as a graph of a branching process. 
Namely, at each internal node, one takes a crossing of the current link diagram and branch on smoothing and flipping the crossing. Hence, to specify a resolution tree, we can describe a rule that determines the corresponding crossings of intermediate link diagrams. The rule determines the evolution of the branching process.

With these ideas in mind, we describe resolution trees introduced in \cite{BIRAL}.

\section{Coherent resolution trees}

This section aims to 
describe two classes of resolution trees introduced in~\cite{BIRAL}. We~call the corresponding trees {\it $\X$-coherent} and {\it $\Y$-coherent}. The description is rather complex, so we start with a simpler construction of {\it descending resolution trees}.

\subsection{Preliminaries on link diagrams}

Let $\D \subseteq \mathbb{R}^2$ be a link diagram. 
Let $x \in \D$, and let $C$ be a Seifert circle of~$\D$. We say that $x$ is a {\it point on~$C$} if $x$ lies in the intersection of~$C$ and~$\D$. We say that $x$ is a {\it point on~$\D$} if $x$ is a point on a Seifert circle of~$\D$. 

Let~$\D$ be a link diagram and $m = \gamma(\D)$. Let $p_1,p_2,\ldots,p_m$ be a sequence of points on~$\D$ such that for $i\neq j$, the points $p_i$ and $p_j$ lie on distinct link components of~$\D$.
We refer to~$p_1,p_2,\ldots,p_m$ as {\it base points}.

We travel through $\D$ by moving along each link component of $\D$ as follows. We start at the first base point $p_1$ and move according to the orientation. As we reach $p_1$ again, we proceeding to the second marked point $p_2$ and start moving according to the orientation. We continue in the same way until one visits the link component containing $p_m$ entirely. This process is called the {\it natural travel} determined by $p_1,p_2,\ldots,p_m$. 

Let $p_1,p_2,\ldots,p_m$ be a sequence of base points on $\D$. 
During 
the corresponding 
natural travel, we visit each crossing of $\D$ exactly twice. A crossing of $\D$ is said to be {\it descending} (resp.\  {\it ascending}) if one travels along the overpassing (resp.\  underpassing) strand first. The diagram~$\D$ is said to be {\it descending} (resp.\  {\it ascending}) if each crossing of~$\D$ is descending (resp.\  ascending). 
We emphasize that these diagrammatic properties depend on the base points.
If $\D$ is descending (resp.\  ascending), then the components of $\D$ are both layered from top to bottom (resp.\  from bottom to top) and represent the unknots. In this case, the diagram~$\D$ represents the unlink with precisely $m$ link components.

Let $x$ be a point on $\D$.
We travel through $\D$ according to the orientation. Denote by $\M(\D; x)$ the longest path starting at $x$ one traveled before meeting either $x$ or
an ascending crossing, that is, a crossing of $\D$ such that one meets its underpassing strand first. We refer to the~path~$\M(\D;x)$ as a {\it maximal descending path} in $\D$ starting at $x$.

We define a {\it maximal ascending path} in $\D$ starting at $x$ similarly.

\subsection{Definition of descending resolution trees}

Let $\D \subseteq \mathbb{R}^2$ be a link diagram. Let us describe all descending resolution trees for $\D$ at once. The~construction of each descending resolution tree $\T$ of $\D$ 
consists of several phases. The resulting tree $\T$ depends on sequences of base points on $\D$.

We start with the one node tree~$\T_0$. At~the end of phase $k$, we obtain a rooted subtree~$\T_k$ of~$\T$. The resulting subtrees satisfy $$\{\D\} = \T_0 \subset \T_1 \subset \ldots \subset \T_m = \T.$$ Thus, one obtains each of the trees from the previous one by extensions shown in Figure~\ref{Resolving}.

In the first phase, we choose an arbitrary base point~$x_1$ on~$\D$. We find the maximal descending path~$\M(\D;x_1)$ on~$\D$ starting at~$x_1$. If~$\M(\D;x_1)$ is closed, that is,~$\M(\D;x_1)$ is the whole link component of~$\D$ containing~$x_1$, then the phase ends. Assume~$\M(\D;x_1)$ is not closed. In~this case, we extend the current tree at a crossing of~$\D$ that is the end of the path~$\M(\D;x_1)$. At this moment, the tree~$\T_1$ consists of three nodes: a parent~$\D$ and its children~$\D^\prime$ and~$\D^{\prime\prime}$ such that $|\D^\prime| = |\D^{\prime\prime}|+1$. Then, we find the maximal descending path~$\M(\D^\prime;x_1)$ on~$\D^\prime$ starting at~$x_1$. If~$\M(\D^\prime;x_1)$ is not closed, then we extend the current tree $\T_1$ similarly by adding children of~$\D^\prime$. We repeat the same procedure for all leaf nodes~$\U$ of the current tree~$\T_1$. At~the end of the first phase, for each leaf node~$\U$ of~$\T_1$, the maximal descending path~$\M(\U;x_1)$ on~$\U$ starting at~$x_1$ is closed. In this case,~$\M(\U;x_1)$ is the whole link component of~$\U$ containing~$x_1$. If~each leaf node of $\T_1$ is a knot diagram, then the construction of~$\T$ ends. Otherwise, we move to the next phase.

\begin{figure}[H]
\centering
\includegraphics[width=9cm]{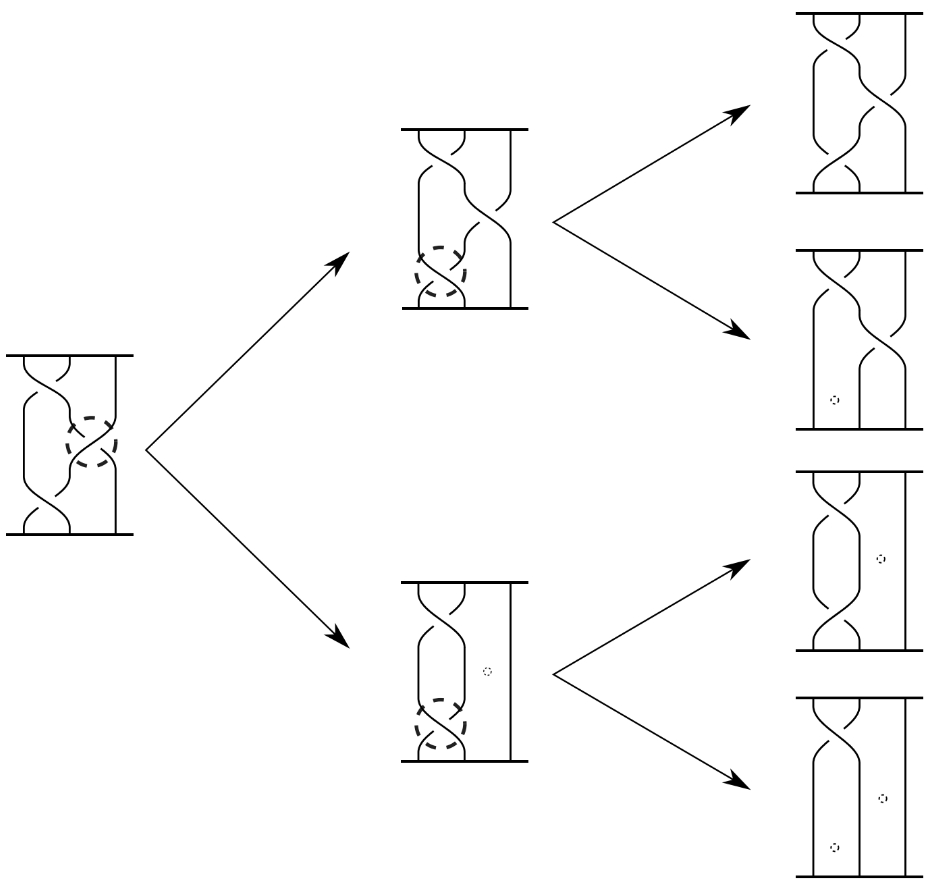}
\caption{An example of a descending resolution tree. For each node of the tree, the base point is 
a point on the top left arc. 
The picture was taken
from \cite{LDH}.}
\label{RTExamaple}
\end{figure}

In the second phase, we extend $\T_1$ as follows. For each leaf node $\V$ of $\T_1$, we choose an arbitrary base point $x_2$ on a link component of $\V$ distinct from that of containing~$x_1$.
We emphasize that the point $x_2$ depends on $\V$.
We find the maximal descending path~$\M(\V;x_2)$ on~$\U$ starting at~$x_2$. If $\M(\V;x_2)$ is not closed, then we extend the current tree at a crossing of $\V$ that is the end of the path $\M(\V; x_2)$. We proceed similarly. The second phase continues until, for each leaf node $\V$ of the current tree $\T_2$, the path $\M(\V;x_2)$ is closed. In this case,~$\M(\V;x_2)$ is the whole link component of~$\V$ containing $x_2$.

We repeat the same procedure until all leaf nodes of the current tree are descending diagrams. Since any descending diagram represents an unlink, the resulting tree $\T$ is a resolution tree.

By varying the base points on the leaf nodes of the trees $\T_0, \T_1, \ldots, \T_{m-1}$, we obtain a few descending resolution trees for $\D$. 

See Figure \ref{RTExamaple} for a descending resolution tree example of the Alexander closure of~$\sigma_1^{-1}\sigma_2\sigma_1^{-1}$. For the sake of simplicity, one presents the nodes as braid diagrams.

The class of descending resolution trees was introduced in \cite[Theorem 2]{Cromwell89}. This class gives an idea of coherent resolution trees.

\subsection{Definition of coherent resolution trees}

We are in the position of describing $\X$-coherent and $\Y$-coherent resolution trees. 
Recall that all link diagrams lie on the oriented plane $\mathbb{R}^2$. 
Thus, one may distinguish clockwise and counterclockwise Seifert circles of $\D$.

Let $\D$ be a link diagram. Let us describe all $\X$-coherent resolution trees for $\D$ at once. The~construction of each $\X$-coherent resolution tree $\T$ of $\D$ consists of several phases. The resulting tree $\T$ depends on sequences of base points on $\D$.

We start with the one node tree~$\T_0$. At~the end of phase $k$, we obtain a rooted subtree~$\T_k$ of~$\T$. The resulting subtrees satisfy $$\{\D\} = \T_0 \subset \T_1 \subset \ldots \subset \T_m = \T.$$ Thus, one obtains each of the trees from the previous one by extensions shown in Figure~\ref{Resolving}.

Following \cite{BIRAL}, we define two rules called {\it descending} and {\it ascending} ones. In the descending rule, one keeps a descending crossing unchanged and branches on flipping and smoothing an ascending one. 
In the ascending rule, one keeps an ascending crossing currently visited and branches on flipping and smoothing a descending one.

At each phase, we follow either the descending or ascending rule. 
By definition, we follow the descending rule if and only if the Seifert circle containing the current base point is clockwise.

More precisely, in the first phase, we choose an arbitrary base point $x_1$ on $\D$. We~find the maximal descending (resp.\  ascending) path in $\D$ starting at $x_1$ whenever the corresponding Seifert circle is clockwise (resp.\  counterclockwise). If the path is not closed, then we extend the current tree at a crossing of $\D$ that is the end of the path. We repeat the same procedure for all leaf nodes $\U$ of the current tree~$\T_1$. At the end of the first phase, for each leaf node~$\U$ of~$\T_1$, the maximal descending (resp.\  ascending) path in~$\U$ starting at~$x_1$ is closed. If each leaf node~$\U$ of~$\T_1$ is a knot diagram, then the construction ends. Otherwise, we move to the next phase.

In the second phase, we extend $\T_1$ as follows. For each leaf node $\V$ of $\T_1$, we choose an arbitrary base point $x_2$ on a link component of $\V$ distinct from that of containing $x_1$.
We emphasize that the point $x_2$ depends on $\V$.
We find the maximal descending (resp.\  ascending) path in~$\V$ starting at $x_2$ whenever the corresponding Seifert circle is clockwise (resp.\  counterclockwise). If the path is not closed, then we extend the current tree at a crossing of~$\V$ that is the end of the path. We continue similarly. The second phase ends when for each leaf node~$\V$ of the current tree $\T_2$, the maximal descending (resp.\  ascending) path in~$\V$ starting at $x_2$ is closed.

We repeat the same procedure until, for each leaf node $\U$ of the current tree, one visits all link components of $\U$. This completes the construction of $\T$.

It is easy to see that for any leaf node $\U$ of $\T$, each link component of $\U$ is either descending or ascending. Also, the link components of $\U$ are stacked over each other. Thus, $\U$ represents an unlink. Therefore, $\T$ is a resolution tree.

By varying the base points on the leaf nodes of the trees $\T_0, \T_1, \ldots, \T_{m-1}$, we obtain a few~$\X$-coherent resolution trees for $\D$.

We define $\Y$-coherent resolution trees similarly. Namely, we follow the descending rule if and only if the Seifert circle containing a current base point is counterclockwise.

\section{Special coherent resolution trees}

To prove Theorem \ref{Theorem2}, we use specific $\X$-coherent and $\Y$-coherent resolution trees referred to as {\it special}. 
In our approach, a class of special coherent resolution trees corresponds to the ordered pair consisting of a link template and a braid placement.
To define these resolution trees, we refer to castle structures for link diagrams introduced in \cite{BIRAL}. The definition of special coherent resolution trees requires specific types of castles that we call {\it appropriate}. In \cite[Lemma~4.3]{BIRAL}, the authors claim that any link diagram admits an appropriate castle. There are several inaccuracies in their proof. Thus, we fix it.

\subsection{Castle structures for link diagrams}

Following \cite{BIRAL}, we describe an additional structure for a link diagram, which is called a {\it castle}. Each castle consists of several segments on Seifert circles, which are called {\it floors}, and several crossings between them, which are called {\it ladders}. 
We endow each floor $F$ with a non-negative integer number referred to as the {\it level of $F$}. By definition, each castle has a unique floor of level $0$ and may have floors of higher levels. 

By definition, any Seifert circle of a diagram $\D$ contains at most one floor of a castle of~$\D$. Thus, each castle represents a rooted subgraph of $\Gamma(\D)$ with a root corresponding to the Seifert circle containing a unique floor of level $0$. The edge set of this graph is the set of all ladders. In these notations, the level of a floor is the distance from the root to the corresponding Seifert circle. 

We identify the plane $\mathbb{R}^2$ with the punctured $2$-sphere $S^2$. Let $\D \subseteq \mathbb{R}^2$ be a link diagram. A Seifert circle $C$ of $\D$ is said to be {\it innermost} if $C$ bounds a closed disk in $S^2$ that contains no Seifert circles of $\D$ in its interior. By definition, the topological type of each castle of~$\D$ is uniquely determined by the ordered pair consisting of an innermost Seifert circle~$C$ of~$\D$ and a base point on~$C$.

\begin{figure}[H]
\centering
\includegraphics[width = 14cm]{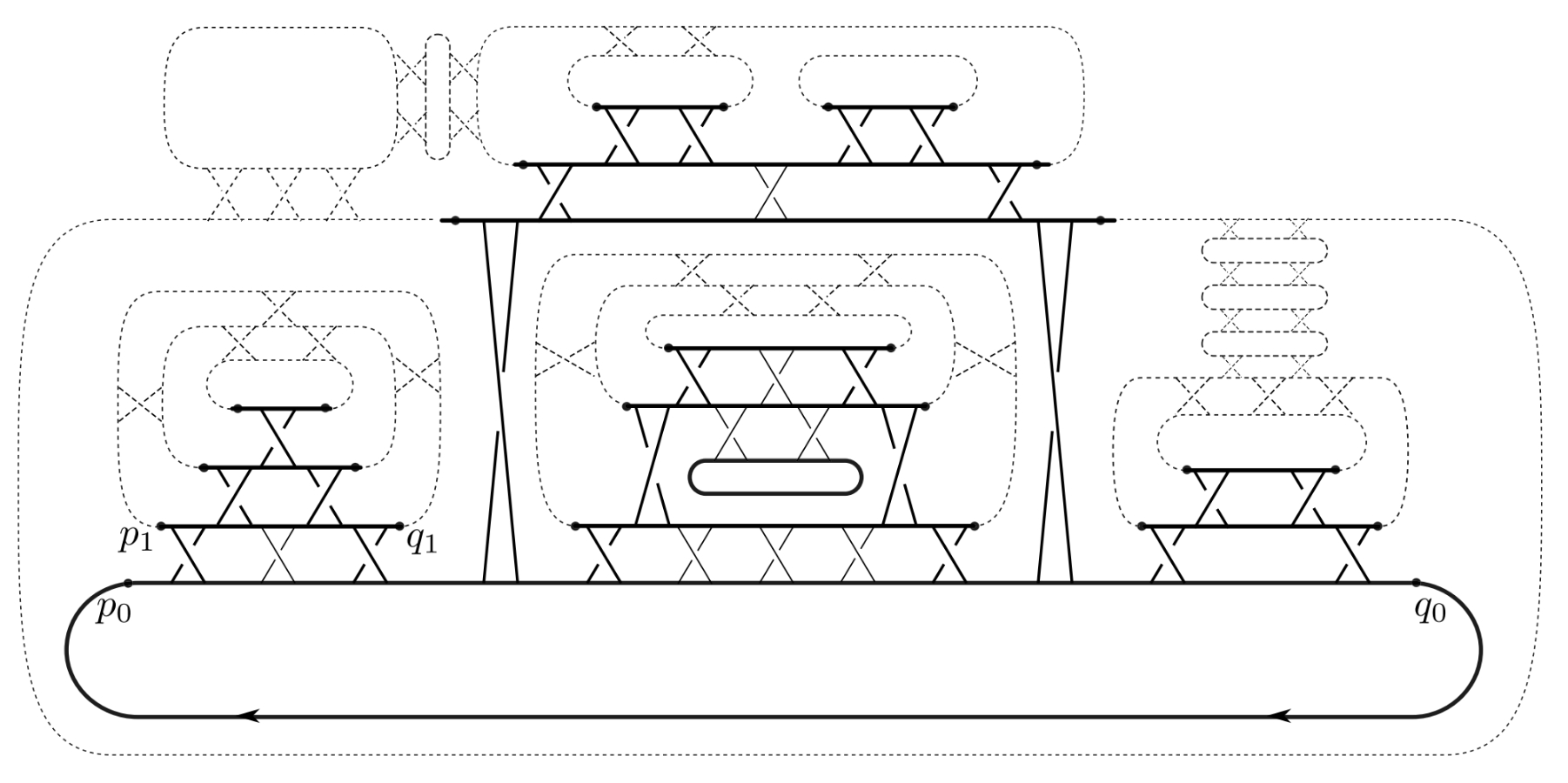}
\caption{An example of a castle. The picture was taken
from \cite{BIRAL}.}
\label{Castle}
\end{figure}

Let $(C,x)$ be the ordered pair consisting of an innermost Seifert circle $C$ of $\D$ and a base point~$x$ on~$C$.
We describe an algorithm for constructing the castle ${\rm Cas}(C,x)$ of~$\D$ determined by~$(C, x)$. We denote by $\Gamma(C,x)$ the rooted subgraph of $\Gamma(\D)$ corresponding to ${\rm Cas}(C,x)$ in the above sense. For clarity, we describe the construction of $\Gamma(C,x)$ in parallel. 

Assume $C$ has no incident crossings. Then the castle ${\rm Cas}(C,x)$ consists only of an arbitrary segment of $C$, and the construction ends. In other words, $\Gamma(C,x)$ consists only of a root~$C$.

Assume $C$ is incident to at least one crossing of~$\D$. Starting at~$x$ and following the orientation of $C$, let us order all crossings incident to~$C$. Since $C$ is innermost, the crossings lie on the same side of $C$. Denote by~$p_0$ and~$q_0$ the boundary points of a segment on $C$ that involves all of the crossings such that $p_0 = x$ and $q_0$ is the point immediately after the last (with respect to the order above) crossing incident to~$C$ (see Figure~\ref{Castle}). This segment of~$C$ is called a {\it floor of level~$0$} of the castle ${\rm Cas}(C,x)$. Recall that by definition, each castle has a unique floor of level $0$.

We are in the position of constructing floors of level $1$. 

Let $C^\prime$ be a Seifert circle of $\D$ that shares at least one crossing with $C$. Starting at~$x$ and following the orientation of $C$, let us order all of them. Let $p_1$ and $q_1$ be two points on~$C^\prime$ immediately before the first such crossing and immediately after the last one, respectively. The segment of~$C^\prime$ starting at $p_1$ and ending at~$q_1$ is called a {\it floor of level~$1$} of the castle~${\rm Cas}(C,x)$. In the same way, we construct floors of level~$1$ for all other Seifert circles of~$\D$ that share at least one crossing with~$C$. In other words, we construct all vertices of~$\Gamma(C,x)$ that locate at a distance of $1$ from the root.

We are in the position of constructing floors of level $2$.

Let $C^\prime$ be a Seifert circle containing a floor $F$ of level $1$.  
Let us construct all floors of level~$2$ incident to $F$. 
Assume $C$ is the only Seifert circle of $\D$ that shares at least one crossing with $F$. Then the construction terminates.
Assume $F$ 
shares crossings
with other Seifert circles of $\D$ other than $C$. Note that such crossings may lie on both sides of~$F$. Let $C^{\prime\prime}$ be a Seifert circle of $\D$ that shares at least one crossing with $F$. 
Following the orientation of $F$, let us order all of them. 
Let~$p_2$ and~$q_2$ be two points on $C^{\prime\prime}$ immediately before the first such crossing and immediately after the last one, respectively. The segment of $C^{\prime\prime}$ starting at~$p_2$ and ending at~$q_2$ is called a {\it floor of level $2$} of the castle~${\rm Cas}(C,x)$. In the same way, we construct floors of level~$2$ for all other Seifert circles of~$\D$ that share at least one crossing with~$F$. In other words, we construct all vertices of~$\Gamma(C,x)$ that are incident to $C^\prime$ and locate at a distance of $2$ from the root.  

We repeat the same procedure for all floors of level $1$, and thus we construct all floors of level~$2$. 

One defines floors of higher levels similarly. Namely, we repeat the same construction
until each floor of maximal level shares at least one crossing with precisely one Seifert circle of $\D$. Eventually, we construct all floors of the castle. A crossing $s$ of $\D$ is called a {\it ladder} if $s$ connects two floors of the castle. By definition, the castle~${\rm Cas}(C,x)$ is the collection of all such floors and ladders. This completes the definitions of both~${\rm Cas}(C,x)$ and~$\Gamma(C,x)$.

We emphasize that there may be crossings of $\D$ that are not ladders of the castle and Seifert circles of $\D$ that do not contain floors of the castle (see dashed objects in Figure~\ref{Castle}). By the construction, the graph obtained from $\Gamma(C,x)$ by identifying all edges that have the same ends is a tree.

\subsection{Appropriate castles}

Let $\D \subseteq \mathbb{R}^2$ be a link diagram, and let $(C,x)$ be the ordered pair consisting of an innermost Seifert circle $C$ of $\D$ and a base point~$x$ on~$C$. Let ${\rm Cas}(C,x)$ be the castle determined by $(C,x)$.

Here and below, we interpret a ladder as a segment that connects two floors. Let $F_1$ and $F_2$ be two floors of ${\rm Cas}(C,x)$ that share at least one ladder.
Following the orientation of $F_1$, let us order all of them. Let $s_1$ and $s_2$ be two adjacent (with respect to the order above) ladders whose ends are $F_1$ and $F_2$. Denote by $F_1^\prime$ and $F_2^\prime$ the subsegments of $F_1$ and $F_2$, respectively, connecting $s_1$ and $s_2$.
Recall that we identify the plane $\mathbb{R}^2$ with the punctured $2$-sphere $S^2$. The circle on $\mathbb{R}^2$ determined by~$F_1^\prime$,~$F_2^\prime$,~$s_1$, and~$s_2$ divides $S^2$ into two closed disks. Denote by $\mathbb{D}$ the one that does not contain the base point $x$.
The disk~$\mathbb{D}$ is called a {\it trap} if its interior contains a floor of the castle ${\rm Cas}(C,x)$.
We say that the ordered pair~$(C,x)$ and the castle~${\rm Cas}(C,x)$ are {\it appropriate} if ${\rm Cas}(C,x)$ has no traps.

For example, a castle shown in Figure \ref{Castle} has two traps.

In \cite[Lemma 4.3]{BIRAL}, the authors claim that any link diagram has an innermost
Seifert circle $C$ and a base point~$x$ on~$C$ such that $(C,x)$ is appropriate. However, their proof contains several inaccuracies.
Below, we specify this result and prove it by using similar arguments. 

We start with preliminaries on link diagrams determined by templates. 

\begin{definition}\label{InducedTemplate}
Let $(\C, \A)$ be a template and $\p$ be a braid placement. Let $\D$ be the link diagram determined by $\p$. 
Suppose one smoothed or flipped some crossings of $\D$ and then deleted some link components. Let $\D^\prime$ be the resulting link diagram. We describe a template~$(\C^\prime, \A^\prime)$ and a braid placement $\p^\prime$ such that the link diagram determined by~$\p^\prime$ is~$\D^\prime$.
To deal with the first two transformations, for each $\a \in \A$, we smooth and flip the corresponding crossings of~$\p(\a)$. 
To deal with the third transformation, given a link component $t$ that one deleted from $\D$ in obtaining $\D^\prime$, we proceed as follows. 
First, we delete from $\left(\bigcup \C \right)\cup \left(\bigcup \A\right)$ the segments of circles in $\C$ that correspond to $t$.
Second, for each arc $\a$, such that $t$ contributes to $\p(\a)$, we delete the strand determined by $t$ from the braid diagram corresponding to $\a$. If there is an arc $\a$ such that $\left\lVert \a \right\rVert=2$, then, after deleting the strand, we delete this arc from the current template.
Third, we rearrange the remaining segments of the circles in $\C$ by moving the ends of these segments along the remaining arcs to obtain a template. 
We repeat the same procedure for all link components that one deleted from $\D$ in obtaining $\D^\prime$. The resulting braid words $\p^\prime(\a)$ are uniquely defined up to far commutativity.
We say that the template~$(\C^\prime, \A^\prime)$ and the braid placement $\p^\prime$ are {\it induced} from $(\C, \A)$ and $\p$, respectively, by the given smoothing, flipping, and link components deletion transformations.
\end{definition}

Let $(\C, \A)$ be a template and $\p$ be a braid placement. Let $\D$ be the link diagram determined by $\p$. Recall that to obtain $\D$, for each $\a \in \A$, one replaces $\a$ in $\left(\bigcup \C \right)\cup \left(\bigcup \A\right)$ by a rectangle and inserts a braid diagram corresponding to~$\p(\a)$ according to the orientation. Following \cite{Nak20}, we refer to these rectangles as {\it braid boxes} of $\D$.

\begin{Lemma}\label{TrappedCastle}
There exists an appropriate pair $(C,x)$ of $\D$ such that $x$ lies in the complement of the braid boxes of $\D$.
\end{Lemma}
\begin{proof}
Let $(C,x)$ be the ordered pair consisting of an innermost Seifert circle $C$ of $\D$ and a base point~$x$ on~$C$ such that $x$ lies in the complement of the braid boxes of $\D$. We suppose that the ordered pair $(C,x)$ is not appropriate and derive from $(C,x)$ new ordered pairs of the above form. After that, we show that a sequence of such derivations leads to an appropriate pair.

Suppose the castle~${\rm Cas}(C,x)$ has at least one trap. Let~$\mathbb{D}$ be any of them. 
Let $F_1$ and~$F_2$ be the floors corresponding to $\mathbb{D}$.
By definition, $\mathbb{D}$ contains at least one floor of the castle~${\rm Cas}(C,x)$. Let $C^\prime$ be a Seifert circle of $\D$ lying within $\mathbb{D}$ such that $C^\prime$ contains a floor sharing crossings with either~$F_1$ or~$F_2$. The circle $C^\prime$ divides $S^2$ into two closed disks. Let $C^{\prime\prime}$ be an arbitrary innermost Seifert circle of~$\D$ lying within the one that does not contain the base point $x$.

Without loss of generality, $C^\prime$ shares crossings with $F_1$. Denote by $F_1^\prime$ the intersection of $F_1$ and the boundary circle of $\mathbb{D}$. We say that the ordered pair $(C^{\prime\prime}, y)$ consisting of the Seifert circle~$C^{\prime\prime}$ of $\D$ and a base point~$y$ on~$C^{\prime\prime}$ is {\it derived from $(C,x)$} if the following conditions hold:
\begin{enumerate}
\item
the point $y$ lies in the complement of the braid boxes of $\D$;
\item
if ${\rm Cas}(C^{\prime\prime},y)$ contains a floor intersecting $F_1$, then this floor lies within $F_1^\prime$. 
\end{enumerate}

We claim that there is at least one point $y$ on $C^{\prime\prime}$ such that the ordered pair $(C^{\prime\prime}, y)$ is derived from $(C,x)$.

Let us prove the claim. Let $y^\prime$ be an arbitrary point on $C^{\prime\prime}$ lying in the complement of the braid boxes of $\D$. If ${\rm Cas}(C^{\prime\prime},y^\prime)$ contains no floors intersecting $F_1$, then there is nothing to prove. Suppose that ${\rm Cas}(C^{\prime\prime},y^\prime)$ contains a floor intersecting $F_1$. 

Since $C^\prime$ separates $C^{\prime\prime}$ from $F_1$, the castle ${\rm Cas}(C^{\prime\prime},y^\prime)$ contains a floor corresponding to~$C^\prime$. Let $C_0, C_1, \ldots, C_m$ be a unique sequence of vertices of $\Gamma(C^{\prime\prime},y^\prime)$ corresponding to the shortest path in $\Gamma(C^{\prime\prime},y^\prime)$ from $C^{\prime\prime}$ to $C^{\prime}$ such that $C_0 = C^{\prime\prime}$ and $C_m = C^\prime$. 

To find at least one point $y$ that satisfies the above conditions, we proceed by induction on $m$. First, suppose $m=0$. In this case, $C^{\prime\prime} = C^\prime$.
Following the orientation of~$F_1^\prime$, let us order all the arcs $\a \in \A$ intersecting both $C_0$ and $F_1^\prime$. Let $y$ be a point on $C_0$ immediately before the first such arc. It is easy to see that the ordered pair~$(C_0, y)$ is derived from~$(C,x)$.

Second, suppose $m \geq 1$. The Seifert circle $C_1$ divides $S^2$ into two closed disks. 
Let us smooth all crossings of the diagram $\D$ that lie within the one that does not contain the base point $x$. Then, let us delete all link components that correspond to Seifert circles lying within the interior of this closed disk. Let~$\D^\prime$ be the resulting link diagram. Let~$(\C^\prime, \A^\prime)$ and~$\p^\prime$ be the template and the braid placement, respectively, induced from $(\C, \A)$ and~$\p$ by these transformations. Note that $\C^\prime$ is a subset of $\C$, $\C^\prime$ contains $C$, and $C_1$ is an innermost Seifert circle of $\D^\prime$. By the induction hypothesis applied to $\D^\prime$, there is at least one point~$y^\prime$ on the Seifert circle $C_1$ of $\D^\prime$ such that the ordered pair $(C_1, y^\prime)$ is derived from $(C,x)$. By moving $y^\prime$ along $C_1$, we can assume that $y^\prime$ lies in the complement of the braid boxes of~$\D$. 

Starting at $y^\prime$ and following the orientation of $C_1$, let us order all the arcs $\a \in \A$ intersecting both $C_0$ and $C_1$ such that $C_0$ and $C_1$ share at least one crossing corresponding to $\a$. 
Let $y$ be a point on $C_0$ immediately before the first such arc. 
Denote by $C_{m+1}$ the Seifert circle of $\D$ corresponding to the floor $F_1$.
By the definition of $y$, any ladder of ${\rm Cas}(C_1,y^\prime)$ that connects $C_1$ and $C_2$ is a ladder of ${\rm Cas}(C_0,y)$. 
It follows that for each~$k \in \{2,3,\ldots,m+1\}$, if the castle ${\rm Cas}(C_0,y)$ has a floor $F$ corresponding to $C_k$, then 
the castle ${\rm Cas}(C_1,y^\prime)$ has a floor containing $F$. Therefore, the ordered pair $(C_0, y)$ is derived from~$(C,x)$. The claim is proved.

It remains to show that by using the derivations described above, one obtains an appropriate pair. Since $\Gamma(\D)$ is bipartite, $C^\prime$ and $F_2$ share no crossings. Thus, the castle~${\rm Cas}(C_0,y)$ has no floors corresponding to $F_2$. Therefore, the derivation decreases the number of floors in the complement of the interior of the current trap.
By the construction, any trap of ${\rm Cas}(C_0,y)$ is either a trap of ${\rm Cas}(C,x)$ or contained within $\mathbb{D}$. Therefore, the new traps are nested, and thus, the derivations lead to a castle without traps. The lemma is proved.
\end{proof}

\subsection{Definition of special coherent resolution trees}

Following \cite{BIRAL}, we define coherent resolution trees referred to as special. 
In our approach, a class of resolution trees corresponds to the ordered pair consisting of a link template and a braid placement.

Let $(\C, \A)$ be a template and $\p$ be a braid placement. Let $\D \subseteq \mathbb{R}^2$ be the link diagram determined by $\p$.

Let us describe all special $\X$-coherent resolution trees that correspond to $(\C, \A, \p)$ at once. As for all $\X$-coherent resolution trees for the diagram $\D$, the construction of each special~$\X$-coherent resolution tree $\T$ consists of several phases. Recall that the resulting tree~$\T$ depends on sequences of base points on~$\D$. To complete the definition, we specify the possible sequences of points.

Recall that we start with the one node tree $\T_0$. At the end of phase $k$, we obtain a rooted subtree $\T_k$ of~$\T$. The resulting subtrees satisfy $$\{\D\} = \T_0 \subset \T_1 \subset \ldots \subset \T_m = \T.$$

Let us specify the construction of $\T_1$. 
Let $(C_1,x_1)$ be an arbitrary appropriate pair of~$\D$
such that $x_1$ lies in the complement of the braid boxes of $\D$.
Let $x_1$ be a base point of the first phase. Then, we construct $\T_1$ as described in the definition of $\X$-coherent resolution trees.
If each leaf node of~$\T_1$ is a knot diagram, then the construction of $\T$ ends. Otherwise, we move to the next phase.

Let us specify the construction of $\T_2$. Let $\V$ be a leaf node of $\T_1$. Recall that $\V$ contains the base point $x_1$. Let $\mathrm{LC}(\V; x_1)$ be the link component of $\V$ containing $x_1$. 
Let us delete~$\mathrm{LC}(\V; x_1)$ from~$\V$.
Let $(\C_1, \A_1)$ and $\p_1$ be the template and the braid placement induced from $(\C,\A)$ and $\p$, respectively, by 
the transformations that one applied in obtaining~$\V$ from~$\D$ and this link component deletion.
We emphasize that $(\C_1, \A_1)$ and $\p_1$ depend on both the leaf node $\V$ and the base point $x_1$ on it.
Denote by $\V\backslash \mathrm{LC}(\V; x_1)$ the link diagram determined by $\p_1$.
Let $(C_2,x_2)$ be an arbitrary appropriate pair of~$\V\backslash \mathrm{LC}(\V; x_1)$ such that~$x_2$ lies in the complement of the braid boxes of $\V\backslash \mathrm{LC}(\V; x_1)$.
Let~$x_2$ be a base point of the second phase concerning $\V$. We follow the same rules concerning the choice of the base point for all other leaf nodes of $\T_1$. 
Then we construct $\T_2$ as described in the definition of~$\X$-coherent resolution trees.
Recall that if for each leaf node~$\V$ of~$\T_2$, one visits all link components of $\V$, then the construction of $\T$ ends. Otherwise, we move to the next phase.

Let us specify the construction of $\T_3$. Let $\U$ be a leaf node of $\T_2$. Recall that $\U$ contains two base points: $x_1$ and $x_2$. Let $\mathrm{LC}(\U;x_2)$ be the link component of $\U$ containing~$x_2$. 
Let us delete both $\mathrm{LC}(\U; x_1)$ and~$\mathrm{LC}(\U; x_2)$ from $\U$.
Let $(\C_2, \A_2)$ and $\p_2$ be the template and the braid placement induced from $(\C,\A)$ and $\p$, respectively, by 
the transformations that one applied in obtaining $\U$ from $\D$ and the two link components deletion.
We emphasize that $(\C_2, \A_2)$ and $\p_2$ depend on both the leaf node $\U$ and the base points $x_1$ and $x_2$ on $\U$.
Denote by $\U \backslash (\mathrm{LC}(\U; x_1) \cup \mathrm{LC}(\U; x_2))$ the link diagram determined by $\p_2$.
Let $(C_3,x_3)$ be an arbitrary appropriate pair of~$\U \backslash (\mathrm{LC}(\U; x_1) \cup \mathrm{LC}(\U; x_2))$ such that $x_3$ lies in the complement of the braid boxes of $\U \backslash (\mathrm{LC}(\U; x_1) \cup \mathrm{LC}(\U; x_2))$.
Let $x_3$ be a base point of the third phase concerning $\U$. We follow the same rules concerning the choice of the base point for all other leaf nodes of $\T_2$. 
Then we construct $\T_3$ as described in the definition of~$\X$-coherent resolution trees.

Recall that we repeat the same procedure until, for each leaf node $\V$ of the current tree, one visits all link components of $\V$.
This completes the construction of $\T$.

Recall there may be several distinct appropriate ordered pairs that satisfy the assumption of~Lemma~\ref{TrappedCastle}. By~varying these ordered pairs, we obtain a few special $\X$-coherent resolution trees for $\D$. We emphasize that the set of special $\X$-coherent resolution trees for $\D$ depends on both the link template $(\C, \A)$ and the braid placement $\p$. 
Lemma \ref{TrappedCastle} implies that for any such pair, there exists at least one special $\X$-coherent resolution tree for $\D$.

We define special $\Y$-coherent resolution trees similarly.

\section{Proof of Theorem \ref{Theorem2}}

This section aims to prove Theorem \ref{Theorem2}. One of the key elements of the proof is inequality \eqref{ThisImpliesMFW}, which we state below.

Let $\U$ be a leaf node of a special $\X$-coherent resolution tree $\T$ for a link diagram $\D$. Recall that $\gamma(\U)$ denotes the number of link components of~$\U$,~$t(\U)$ denotes the number of crossings of~$\D$ that one smoothed in obtaining~$\U$, and~$t^-(\U)$ denotes the number of negative crossings among the smoothed ones.
Besides, recall that the skein polynomial~$\P(\D; a,z)$ admits a decomposition as a sum indexed by all leaf vertices of~$\T$ such that the term of the sum that corresponds to $\U$ is 
\begin{align*}
(-1)^{t^-(\U)} z^{t(\U)} a^{\omega(\U) - \omega(\D)} ((a-a^{-1})z^{-1})^{\gamma(\U)-1}.
\end{align*}
In \cite[Section 5]{BIRAL}, the authors show that 
the following inequality holds:
\begin{align}\label{ThisImpliesMFW}
\omega(\U) - \omega(\D) + \gamma(\U) - 1 \leq -\omega(\D) + s(\D)-1.
\end{align}
A similar result concerning the lower $a$-degree bound of~$\P(\D; a,z)$ that a leaf node contributes to holds for special~$\Y$-coherent resolution trees. In particular, these results imply the Morton--Franks--Williams inequality \eqref{DegSeif0}.

We say that a leaf node $\U$ of $\T$ {\it contributes to the highest~$a$-degree term} if \eqref{ThisImpliesMFW} is sharp.

We are in the position of proving Theorem \ref{Theorem2}.

Let $(\C, \A)$ be a knitted template, and let $\p$ be a braid placement. Let $\D$ be the link diagram determined by $\p$. Let $E$ (resp.\  $e$) be the highest (resp.\  the lowest) $a$-degree of the skein polynomial~$\P(\D; a,z)$ of $\D$. 
We aim to prove that if for all $\a \in \A$, the braid word $\p(\a)$ is locally $-$twisted, then~$E = -\omega(\D) + s(\D)-1$. Similar arguments show that if for all $\a \in \A$, the braid word $\p(\a)$ is locally $+$twisted, then~$e = -\omega(\D) - s(\D) + 1$. This implies that the Morton--Franks--Williams inequality is sharp for any locally twisted link diagram.

Suppose that for all $\a \in \A$, the braid word $\p(\a)$ is locally $-$twisted. Let $\T$ be an arbitrary special $\X$-coherent resolution tree for $\D$.

We are in the position of proving equality $E = -\omega(\D) + s(\D)-1$. 
The proof is in three steps. 
In the first step, 
we describe a specific leaf node $\U^\ast$ of $\T$ such that:
\begin{enumerate}
\item[{\rm(i)}] one smoothed all positive crossings of $\D$ in obtaining $\U^\ast$;
\item[{\rm(ii)}] the leaf node $\U^\ast$ contributes to the highest $a$-degree term.
\end{enumerate}

In the second step, 
we show that if a leaf node $\U$ of $\T$ contributes to the highest $a$-degree term, then~$\gamma(\U) = s(\D)$, $\omega(\U) = 0$, and~$t^-(\U) \leq t^-(\U^\ast)$. 
In~the third step, we show that the previous ones imply the result.

\subsection{First step}

We start with the following observation.

\begin{Lemma}\label{Fragment}
Let $r := (-1, -1, \ldots, -1) \in \{1,-1\}^{n-1}$. Suppose an $r$-homogeneous braid word~$u$ represents~$\Delta_{1,n}^{-1}$. Then there exist braid words of the form
$u_1 \sigma_1^{-1}\sigma_2^{-1}\ldots \sigma_{n-1}^{-1} u_2$ and~$u_3 \sigma_{n-1}^{-1}\sigma_{n-2}^{-1}\ldots \sigma_{1}^{-1} u_4$ lying in the same far commutativity class as of the braid word $u$ such that~$u_2, u_3 \in \{\sigma_1^{-1}, \sigma_2^{-1},\ldots, \sigma_{n-2}^{-1}\}^\ast$ and~$u_1, u_4 \in \{\sigma_2^{-1}, \sigma_3^{-1}, \ldots, \sigma_{n-1}^{-1}\}^\ast$. 
\end{Lemma}

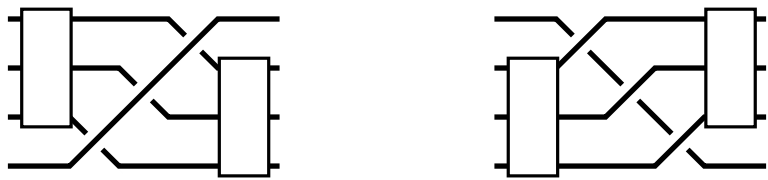
\begin{figure}[H]
\centering
\begin{tikzpicture}[rotate=90, scale=0.65, every node/.style={scale=0.65}]
\pic[
  rotate=90,
  line width=2pt,
  braid/control factor=0,
  braid/nudge factor=0,
  braid/gap=0.11,
  braid/number of strands = 4,
  name prefix=braid,
] at (0,0) {braid={
1 s_1 s_2 s_3 1
}};
\draw[draw=black, very thick, fill=white] (1-0.2,-0.28) rectangle ++(2.4,-1);
\draw[draw=black, very thick, fill=white] (-0.2,-4.28) rectangle ++(2.4,-1);
\end{tikzpicture} \hspace{2.5cm}
\begin{tikzpicture}[rotate=90, scale=0.65, every node/.style={scale=0.65}]
\pic[
  rotate=90,
  line width=2pt,
  braid/control factor=0,
  braid/nudge factor=0,
  braid/gap=0.11,
  braid/number of strands = 4,
  name prefix=braid,
] at (0,0) {braid={
1 s_3 s_2 s_1 1
}};
\draw[draw=black, very thick, fill=white] (-0.2,-0.28) rectangle ++(2.4,-1);
\draw[draw=black, very thick, fill=white] (1-0.2,-4.28) rectangle ++(2.4,-1);
\end{tikzpicture}
\caption{An illustration for Lemma \ref{Fragment}.}
\label{Illustration}
\end{figure}

\begin{proof}
We show that by using the moves corresponding to the far commutativity relations, one may transform $u$ into a braid word of the form $u_1 \sigma_1^{-1}\sigma_2^{-1}\ldots \sigma_{n-1}^{-1} u_2$. For the second braid word, the argument is similar.

First, the assertion of Lemma \ref{Fragment} holds for the braid word \eqref{DeltaWord1}.

Second, any two $r$-homogeneous braid words that represent the same braid are related by the diagram transformations corresponding to the braid relations and the far commutativity relations (see \cite[Theorem 9.2.5]{WPG}). 
It is easy to check that if one obtains a braid word $v_2$ from $v_1$ via a single braid relation move, the assertion holds for $v_1$ if and only if it holds for $v_2$. Similar holds for a single far commutativity relation move. Therefore, the assertion of Lemma \ref{Fragment} holds for $u$. The lemma is proved.
\end{proof}

Denote by $[w]$ the braid diagram corresponding to a braid word $w$. We refer to the vertical segments obtained by smoothing all crossings of 
$[w]$ as {\it Seifert segments} of~$[w]$. We order the Seifert segments from left to right.
Recall that, given $\a \in \A$, the symbol $\left\lVert \a \right\rVert$ denotes the number of strands of $[\p(\a)]$.

We are in the position of describing a specific leaf node $\U^\ast$ of $\T$. Note that for any leaf node $\V$ of $\T$, there is a unique path from $\D$ to $\V$. This path determines a sequence of diagram transformations leading from $\D$ to $\V$. Each of these transformations is either smoothing or flipping of a crossing that is the end of either a maximal descending or maximal ascending path. Therefore, to determine a specific leaf vertex of $\T$, one can travel naturally through $\D$ by moving along each link component and describe a sequence of diagram transformations of the above form. 

Let $(C_1,x_1)$ be the appropriate pair that one used in the first phase to construct~$\T$. Let us start at $x_1$ and move according to the orientation. Assume at the moment that $C_1$ intersects an arc of the template. 
Let $\a \in \A$ be the first arc such that one encountered the corresponding braid box of $\D$. 

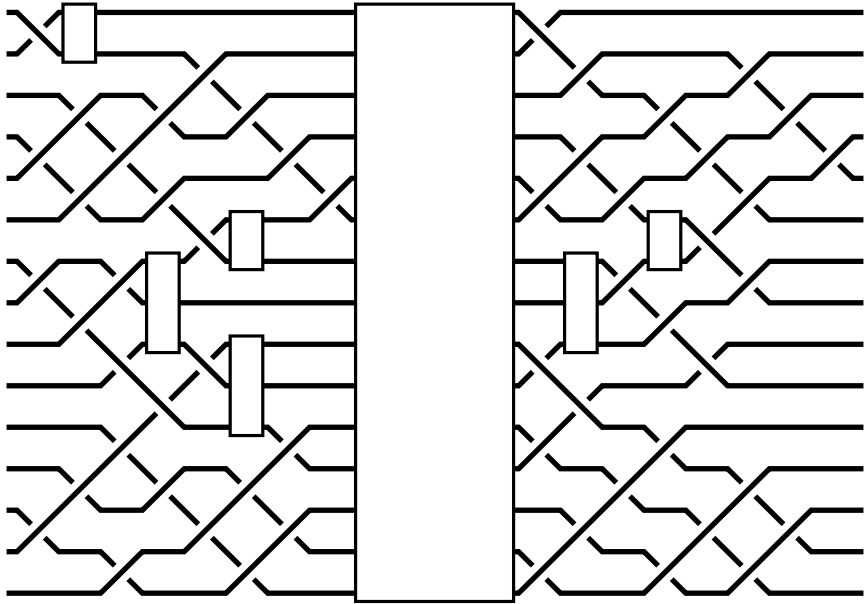
\begin{figure}
\centering
\begin{tikzpicture}[rotate=90,scale=0.55, every node/.style={scale=0.55}]
\pic[
  rotate=90,
  line width=2pt,
  braid/control factor=0,
  braid/nudge factor=0,
  braid/gap=0.11,
  braid/number of strands = 15,
  name prefix=braid,
] at (0,0) {braid={
s_2-s_8-s_{11}-s_{14}^{-1}
s_3-s_7-s_{10}-s_{12}
s_1-s_4-s_6^{-1}-s_8-s_{11}
s_3-s_5^{-1}-s_{10}-s_{12}
s_2-s_6^{-1}-s_9^{-1}-s_{13}
s_1-s_3-s_{12}
s_2-s_4-s_{11}
s_{10}
1 1 1 1 
s_1-s_4-s_6^{-1}-s_{10}-s_{14}^{-1}
s_2-s_5^{-1}-s_{11}-s_{13}
s_3-s_8-s_{10}
s_1-s_4-s_7-s_{12}
s_2-s_6^{-1}-s_9^{-1}-s_{11}
s_1-s_3-s_8-s_{10}-s_{13}
s_2-s_{12}
s_{11}
}};
\draw[draw=black, very thick, fill=white] (4-0.2,-5.35) rectangle ++(2.4,-0.78);
\draw[draw=black, very thick, fill=white] (6-0.2,-3.35) rectangle ++(2.4,-0.78);
\draw[draw=black, very thick, fill=white] (8-0.2,-5.35) rectangle ++(1.4,-0.78);
\draw[draw=black, very thick, fill=white] (13-0.2,-1.35) rectangle ++(1.4,-0.78);
\draw[draw=black, very thick, fill=white] (-0.2,-8.35) rectangle ++(14.4,-3.78);
\draw[draw=black, very thick, fill=white] (6-0.2,-13.35) rectangle ++(2.4,-0.78);
\draw[draw=black, very thick, fill=white] (8-0.2,-15.35) rectangle ++(1.4,-0.78);
\end{tikzpicture}
\caption{An example of a locally twisted braid diagram. The rectangles can be filled by arbitrary braid diagrams so that the resulting one is homogeneous.}
\label{ExampleD}
\end{figure}

\begin{spacing}{1.5}
\end{spacing}

{\it Case 1}. The first crossing one encounters is negative.
This crossing has the form either~$\sigma_1^{-1}$ or~$\sigma_{\left\lVert \a \right\rVert-1}^{-1}$ depending on the direction of $C_1$. Suppose $C_1$ is clockwise. Then the first crossing has the form $\sigma_1^{-1}$. Recall that in this case, one has to follow the descending rule. We have no choice but to keep the first crossing unchanged (and hence cross it). We continue traveling along the second Seifert segment of the braid diagram~$[\p(a)]$.

We assume at the moment that $i_2>2$. Recall that~$\p(\a)$ is locally $-$twisted. 
We wait for a crossing of the form $\sigma_2^{-1}$. 
The latter exists due to Lemma~\ref{Fragment}.
When such a crossing appears, we have no choice but to keep it unchanged. We continue traveling along the third Seifert segment of~$[\p(\a)]$ and wait for a crossing of the form~$\sigma_3^{-1}$. The latter exists due to Lemma~\ref{Fragment}. This~process continues until we cross a crossing of the form $\sigma_{i_2-1}^{-1}$.

Let $q$ be the last appearance
of the letter $\sigma_{i_2-1}^{-1}$ in~$\p(\a)$ such that $\sigma_{i_2-1}^{-1}\sigma_{i_2-2}^{-1}\ldots \sigma_{3}^{-1}\sigma_{2}^{-1}\sigma_{1}^{-1}$ is a subsequence of the suffix of~$\p(\a)$ that begins at $q$. The latter exists due to Lemma~\ref{Fragment}. We continue traveling along the Seifert segment of index $i_2$ of~$[\p(\a)]$. We smooth all crossings of the forms $\sigma_{i_2-1}^{-1}$ and $\sigma_{i_2}$ and wait for the crossing corresponding to $q$. When this crossing appears, we flip it and hence move to the Seifert segment of index $i_2-1$ of~$[\p(\a)]$.

Let $q^\prime$ be the last appearance
of the letter $\sigma_{i_2-2}^{-1}$ in~$\p(\a)$ such that $\sigma_{i_2-2}^{-1}\sigma_{i_2-3}^{-1}\ldots \sigma_{3}^{-1}\sigma_{2}^{-1}\sigma_{1}^{-1}$ is a subsequence of the suffix of~$\p(\a)$ that begins at $q^\prime$. The latter exists due to Lemma~\ref{Fragment}. We continue traveling along the Seifert segment of index $i_2-1$ of~$[\p(\a)]$. 
We
wait for the crossing corresponding to $q^\prime$. When this crossing appears, we flip it and hence move to the Seifert segment of index $i_2-2$ of~$[\p(\a)]$. This process continues until we flip a crossing of the form $\sigma_{1}^{-1}$ and hence move to the ending point of the first Seifert segment of~$[\p(\a)]$. The case $i_2=2$ is similar.

Roughly speaking, we extricate the first strand of the braid diagram.

If $C_1$ is counterclockwise, the construction is the same after one replaces the descending rule by the ascending one. The first case is complete.

\begin{spacing}{1.5}
\end{spacing}

{\it Case 2}. The first crossing one encounters is positive.
This crossing has the form either~$\sigma_1$ or~$\sigma_{\left\lVert \a \right\rVert-1}$ depending on the direction of $C_1$. Suppose~$C_1$ is clockwise. Then the first crossing has the form $\sigma_1$. Recall that in this case, one has to follow the descending rule. We smooth the first crossing and continue traveling along the first Seifert segment of~$[\p(\a)]$. By repeating this process, we arrive at the end of the first Seifert segment of~$[\p(\a)]$ by smoothing all the crossings of the form $\sigma_1$. 

If $C_1$ is counterclockwise, then one has to follow the ascending rule. We smooth all crossings of the form $\sigma_{\left\lVert \a \right\rVert-1}$ similarly. The second case is complete.

\begin{spacing}{1.5}
\end{spacing}

After the first arc $\a$ passed, we continue moving along $C_1$ according to the orientation. 
We follow the same rules as above and thus extend the path in $\T$ until the link component that contains $x_1$ is descending (resp.\ ascending) provided $C_1$ is clockwise (resp.\ counterclockwise).
Eventually, we obtain a leaf node~$\U_1^\ast$ of~$\T_1$. Recall that to find the second base point of~$\U_1^\ast$, we delete the link component~$\mathrm{LC}(\U_1^\ast; x_1)$ of~$\U_1^\ast$ containing~$x_1$. 
Let $(\C_1, \A_1)$ and~$\p_1$ be the template and the braid placement induced from $(\C,\A)$ and~$\p$, respectively, by 
the transformations that one applied in obtaining $\U_1^\ast$ from $\D$ and this link component deletion. Let $\U_1^\ast\backslash \mathrm{LC}(\U_1^\ast; x_1)$ be the link diagram determined by $\p_1$. Since the template~$(\C,\A)$ is knitted, $(\C_1, \A_1)$ is knitted too. 
By using Lemma \ref{Fragment}, it is easy to check that for all $\a \in \A_1$, the braid word $\p_1(\a)$ is locally $-$twisted. 
Therefore, we can repeat the same procedure for $\U_1^\ast\backslash \mathrm{LC}(\U_1^\ast; x_1)$ until the construction of~$\U^\ast$ terminates.

We are in the position of proving assertions {\rm(i)} and {\rm(ii)} concerning $\U^\ast$. 

On the one hand, by definition of coherent resolution trees, one visits each crossing of~$\D$ exactly twice (and hence at least once). On the other hand, by the definition of~$\U^\ast$, one smoothed all positive crossings one encountered. Therefore, one smoothed all positive crossings of $\D$ in obtaining~$\U^\ast$. Therefore, assertion {\rm(i)} holds. 

It remains to show that~$\U^\ast$ contributes to the highest $a$-degree term. By the construction of $\U^\ast$, one has~$\gamma(\U^\ast) = s(\D)$. 
Recall that the link components of any leaf node of $\T$ are stacked over each other. Thus, the contribution of crossings between any two distinct link components of $\U^\ast$ to the writhe~$\omega(\U)$ is zero. Note that each link component of $\U^\ast$ represents a simple closed curve on the plane. Therefore, one has $\omega(\U^\ast) = 0$. Hence,
\begin{align*}
\omega(\U^\ast) - \omega(\D) + \gamma(\U^\ast) - 1 = -\omega(\D) + s(\D)-1,    
\end{align*}
and thus assertion {\rm(ii)} holds. The first step is complete.

\begin{figure}
\centering
\begin{tikzpicture}[rotate=90,scale=0.55, every node/.style={scale=0.55}]
\pic[
  rotate=90,
  line width=2pt,
  braid/control factor=0,
  braid/nudge factor=0,
  braid/gap=0.11,
  braid/number of strands = 15,
  name prefix=braid,
] at (0,0) {braid={
s_2-s_8-s_{11}
s_3-s_7-s_{10}-s_{12}
s_1-s_4-s_8-s_{11}
s_3-s_{10}-s_{12}
s_2-s_{13}
s_1-s_3-s_{12}
s_2-s_4-s_{11}
s_{10}
1 1 1 1 
s_1^{-1}-s_4^{-1}-s_{10}^{-1}
s_2^{-1}-s_{11}^{-1}-s_{13}^{-1}
s_3^{-1}-s_8^{-1}-s_{10}^{-1}
s_1^{-1}-s_4^{-1}-s_7^{-1}-s_{12}^{-1}
s_2^{-1}-s_{11}^{-1}
s_1^{-1}-s_3^{-1}-s_8^{-1}-s_{10}^{-1}-s_{13}^{-1}
s_2^{-1}-s_{12}^{-1}
s_{11}^{-1}
}};
\draw[thick, densely dotted] (13.5,-0.7) circle (0.1cm);
\draw[thick, densely dotted] (13.5,-1.7) circle (0.1cm);
\draw[thick, densely dotted] (8.5,-4.7) circle (0.1cm);
\draw[thick, densely dotted] (8.5,-5.7) circle (0.1cm);
\draw[thick, densely dotted] (7.5,-3.7) circle (0.1cm);
\draw[thick, densely dotted] (6.5,-3.7) circle (0.1cm);
\draw[thick, densely dotted] (5.5,-2.7) circle (0.1cm);
\draw[thick, densely dotted] (4.5,-3.7) circle (0.1cm);
\draw[thick, densely dotted] (5.5,-4.7) circle (0.1cm);
\draw[thick, densely dotted] (5.5,-5.7) circle (0.1cm);
\draw[thick, densely dotted] (4.5,-5.7) circle (0.1cm);
\draw[thick, densely dotted] (0.5,-8.7) circle (0.1cm);
\draw[thick, densely dotted] (0.5,-9.7) circle (0.1cm);
\draw[thick, densely dotted] (0.5,-10.7) circle (0.1cm);
\draw[thick, densely dotted] (0.5,-11.7) circle (0.1cm);
\draw[thick, densely dotted] (1.5,-8.7) circle (0.1cm);
\draw[thick, densely dotted] (1.5,-9.7) circle (0.1cm);
\draw[thick, densely dotted] (1.5,-10.7) circle (0.1cm);
\draw[thick, densely dotted] (1.5,-11.7) circle (0.1cm);
\draw[thick, densely dotted] (2.5,-8.7) circle (0.1cm);
\draw[thick, densely dotted] (2.5,-9.7) circle (0.1cm);
\draw[thick, densely dotted] (2.5,-10.7) circle (0.1cm);
\draw[thick, densely dotted] (2.5,-11.7) circle (0.1cm);
\draw[thick, densely dotted] (3.5,-8.7) circle (0.1cm);
\draw[thick, densely dotted] (3.5,-9.7) circle (0.1cm);
\draw[thick, densely dotted] (3.5,-10.7) circle (0.1cm);
\draw[thick, densely dotted] (3.5,-11.7) circle (0.1cm);
\draw[thick, densely dotted] (4.5,-8.7) circle (0.1cm);
\draw[thick, densely dotted] (4.5,-9.7) circle (0.1cm);
\draw[thick, densely dotted] (4.5,-10.7) circle (0.1cm);
\draw[thick, densely dotted] (4.5,-11.7) circle (0.1cm);
\draw[thick, densely dotted] (5.5,-8.7) circle (0.1cm);
\draw[thick, densely dotted] (5.5,-9.7) circle (0.1cm);
\draw[thick, densely dotted] (5.5,-10.7) circle (0.1cm);
\draw[thick, densely dotted] (5.5,-11.7) circle (0.1cm);
\draw[thick, densely dotted] (6.5,-8.7) circle (0.1cm);
\draw[thick, densely dotted] (6.5,-9.7) circle (0.1cm);
\draw[thick, densely dotted] (6.5,-10.7) circle (0.1cm);
\draw[thick, densely dotted] (6.5,-11.7) circle (0.1cm);
\draw[thick, densely dotted] (7.5,-8.7) circle (0.1cm);
\draw[thick, densely dotted] (7.5,-9.7) circle (0.1cm);
\draw[thick, densely dotted] (7.5,-10.7) circle (0.1cm);
\draw[thick, densely dotted] (7.5,-11.7) circle (0.1cm);
\draw[thick, densely dotted] (8.5,-8.7) circle (0.1cm);
\draw[thick, densely dotted] (8.5,-9.7) circle (0.1cm);
\draw[thick, densely dotted] (8.5,-10.7) circle (0.1cm);
\draw[thick, densely dotted] (8.5,-11.7) circle (0.1cm);
\draw[thick, densely dotted] (9.5,-8.7) circle (0.1cm);
\draw[thick, densely dotted] (9.5,-9.7) circle (0.1cm);
\draw[thick, densely dotted] (9.5,-10.7) circle (0.1cm);
\draw[thick, densely dotted] (9.5,-11.7) circle (0.1cm);
\draw[thick, densely dotted] (10.5,-8.7) circle (0.1cm);
\draw[thick, densely dotted] (10.5,-9.7) circle (0.1cm);
\draw[thick, densely dotted] (10.5,-10.7) circle (0.1cm);
\draw[thick, densely dotted] (10.5,-11.7) circle (0.1cm);
\draw[thick, densely dotted] (11.5,-8.7) circle (0.1cm);
\draw[thick, densely dotted] (11.5,-9.7) circle (0.1cm);
\draw[thick, densely dotted] (11.5,-10.7) circle (0.1cm);
\draw[thick, densely dotted] (11.5,-11.7) circle (0.1cm);
\draw[thick, densely dotted] (12.5,-8.7) circle (0.1cm);
\draw[thick, densely dotted] (12.5,-9.7) circle (0.1cm);
\draw[thick, densely dotted] (12.5,-10.7) circle (0.1cm);
\draw[thick, densely dotted] (12.5,-11.7) circle (0.1cm);
\draw[thick, densely dotted] (13.5,-8.7) circle (0.1cm);
\draw[thick, densely dotted] (13.5,-9.7) circle (0.1cm);
\draw[thick, densely dotted] (13.5,-10.7) circle (0.1cm);
\draw[thick, densely dotted] (13.5,-11.7) circle (0.1cm);
\draw[thick, densely dotted] (13.5,-12.7) circle (0.1cm);
\draw[thick, densely dotted] (8.5,-15.7) circle (0.1cm);
\draw[thick, densely dotted] (8.5,-16.7) circle (0.1cm);
\draw[thick, densely dotted] (7.5,-13.7) circle (0.1cm);
\draw[thick, densely dotted] (6.5,-13.7) circle (0.1cm);
\draw[thick, densely dotted] (5.5,-12.7) circle (0.1cm);
\draw[thick, densely dotted] (4.5,-13.7) circle (0.1cm);
\draw[thick, densely dotted] (5.5,-16.7) circle (0.1cm);
\end{tikzpicture}
\caption{The braid diagram obtained from those in Figure \ref{ExampleD} in the construction of $\U^\ast$.}
\label{ExampleD1}
\end{figure}
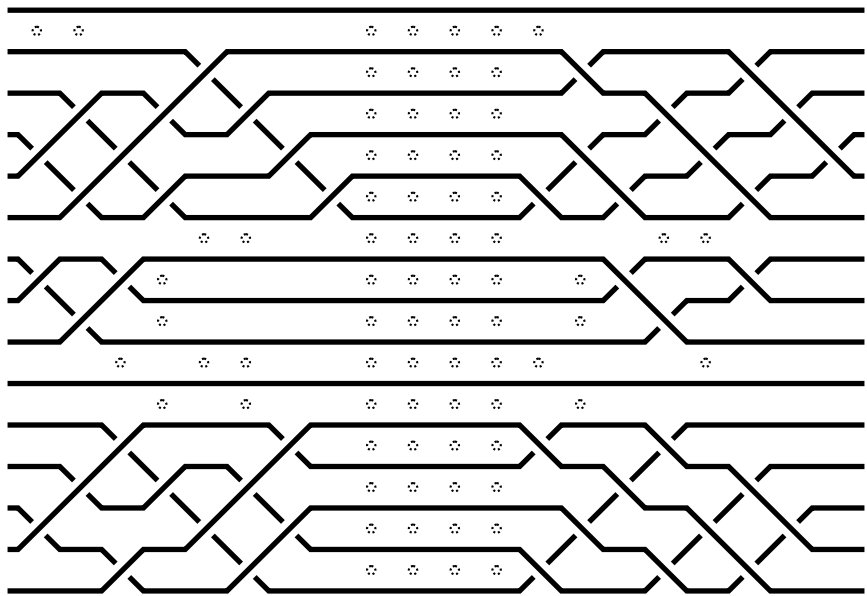

\subsection{Second step}

Let $\U$ be a leaf node of $\T$. We aim to prove that if $\U$ contributes to the highest $a$-degree term, then:
\begin{enumerate}
\item[{\rm(i)}] one has $\gamma(\U) = s(\D)$ and $\omega(\U) = 0$;
\item[{\rm(ii)}] inequality $t^-(\U) \leq t^-(\U^\ast)$ holds.
\end{enumerate}
Let $P$ be a unique path from the root of~$\T$ to~$\U$. Given~$i \in \{0,1,\ldots,\gamma(\U)\}$, let~$\U_i$ be a unique leaf of $P \cap \T_i$. In particular,~$\U_0 = \D$ and~$\U_{\gamma(\U)} = \U$. Recall that each~$\U_i \subseteq \mathbb{R}^2$ contains~$i$ base points, which we denote by $x_1, x_2, \ldots, x_{i}$. 

Let $i \in \{0,1,\ldots,\gamma(\U)\}$. Recall that $\mathrm{LC}(\U_i; x_i)$ denotes the link component of~$\U_i$ that contains $x_i$. This link component determines a curve on the plane $\mathbb{R}^2$. By the construction of~$\T$, to obtain~$\U_i$, one smooths all crossings of~$\U_{i-1}$ incident to~$\mathrm{LC}(\U_i; x_i)$ and then flips several crossings that~$\mathrm{LC}(\U_i;x_i)$ intersects. Let 
\begin{align*}
\U^{i-1} := \U_{i-1} \backslash (\mathrm{LC}(\U_{i-1};x_1) \cup \mathrm{LC}(\U_{i-1};x_2) \cup \ldots \cup \mathrm{LC}(\U_{i-1};x_{i-1}))    
\end{align*}
be the link diagram obtained from~$\U_{i-1}$ by deleting each of the link components~$\mathrm{LC}(\U_{i-1};x_1), \mathrm{LC}(\U_{i-1};x_2), \ldots, \mathrm{LC}(\U_{i-1};x_{i-1})$. 

Let~$(C_i, x_i)$ be the appropriate pair of the diagram $\U^{i-1}$
that one uses in the construction of $\U_{i}$. 
Let $t$ be the maximal segment of the plane curve determined by the link component~$\mathrm{LC}(\U_i; x_i)$ such that $t$ contains within the castle ${\rm Cas}(C_i, x_i)$ of $\U^{i-1}$ and the starting point of $t$ is $x_i$. Let $L_i$ be the level of a floor of ${\rm Cas}(C_i, x_i)$ whose ending point coincides with the ending point of $t$.

Lemma \ref{Ground1} and Corollary \ref{Ground2} below are Lemma~5.2 and Corollary~5.3, respectively, in~\cite{BIRAL}.

\begin{Lemma}\label{Ground1}
If the leaf node $\U$ contributes to the highest~$a$-degree term, then for each index~$i \in \{1,2,\ldots,\gamma(\U)\}$, one has~$L_i=0$.
\end{Lemma}

\begin{Corollary}\label{Ground2}
If the leaf node $\U$ contributes to the highest~$a$-degree term, then each link component of $\U$ represents a simple closed curve on 
the plane.
\end{Corollary}

We are in the position of proving assertion {\rm(i)}.

\begin{Corollary}
If the leaf node $\U$ contributes to the highest~$a$-degree term, then one has~$\gamma(\U) = s(\D)$ and $\omega(\U) = 0$.
\end{Corollary}
\begin{proof}
Lemma \ref{Ground1} implies that for each $i \in \{1,2,\ldots,\gamma(\U)\}$, one has $s(\U^i) = s(\U^{i-1}) - 1.$
Therefore, the number of link components of $\U$ is $s(\D)$. 

By the construction of $\X$-coherent resolution trees, the link components of $\U$ are stacked over each other. Thus, the contribution of crossings between any two distinct link components of $\U$ to the writhe $\omega(\U)$ is zero. Thus, Corollary~\ref{Ground2} implies $\omega(\U) = 0$. The corollary is proved.
\end{proof}

We are in the position of proving assertion {\rm(ii)}. Suppose the leaf node $\U$ of $\T$ contributes to the highest $a$-degree term. Recall that $(\C, \A)$ and $\p$ denote the knitted template and the braid placement, respectively, such that the link diagram determined by $\p$ is $\D$. We need the following auxiliary result.

\begin{Lemma}\label{MainLemma}
Let $i \in \{1,2,\ldots,\gamma(\U)\}$ and $\a \in \A$. 
Suppose the link component $\mathrm{LC}(\U; x_i)$ of $\U$ contributes a strand~$s(\a,i)$ to the braid diagram $[\p(\a)]$.
Then the boundary points of~$s(\a,i)$ lie on the same Seifert segment of~$[\p(\a)]$.
\end{Lemma}
\begin{proof}
First, suppose $i = 1$. Let~$(C_1, x_1)$ be the appropriate pair of the diagram $\U^0 = \D$ that one uses in the construction of $\U_{1}$. Starting at $x_1$ and following the orientation of $C_1$, let us order the arcs in $\A$ intersecting $C_1$ as follows: $\a_1, \a_2, \ldots, \a_r$. 

Let $d \in \{1,2,\ldots,\left\lVert \a_1 \right\rVert\}$ be the index of the ending point of $s(\a_1,1)$. We claim that~$d = 1$. 

Assume the converse, that is, $d_1>1$. 
Let $F_0$ be a unique floor of level~$0$ of the castle~${\rm Cas}(C_1, x_1)$, and let $p_0 = x_1$ and $q_0$ be the boundary points of $F_0$ (see Figure \ref{TrappedCase}).
Let~$F_1$ be the floor of level~$1$ of~${\rm Cas}(C_1, x_1)$ that contains the second Seifert segment of the braid diagram~$[\p(\a_1)]$.

\begin{figure}[H]
\centering
\begin{tikzpicture}[scale=0.95, every node/.style={scale=0.95}]
\draw[very thick,draw=red] (0,0) -- (1,0);
\draw[very thick] (1,0) -- (14,0);
\draw[very thick,draw=red] (14,0) -- (15,0);
\filldraw[red] (0.3,0) circle (1pt) node[anchor=south] {$p_0$};
\filldraw[red] (14.7,0) circle (1pt) node[anchor=south] {$q_0$};
\draw[very thick] (1-0.2,0.5) -- (2+0.2,0.5);
\draw[very thick] (1-0.2,1) -- (2+0.2,1);
\draw[very thick] (1-0.2,1.5) -- (2-0.2,1.5);
\draw[very thick,draw=red] (2-0.2,1.5) -- (2+0.2,1.5);
\draw[very thick] (1-0.2,2) -- (2+0.2,2);
\draw[draw=black, very thick, fill=white] (1,-0.2) rectangle ++(1,2.4);
\draw[very thick] (4-0.2,0.5) -- (5+0.2,0.5);
\draw[draw=black, very thick, fill=white] (4,-0.2) rectangle ++(1,0.9);
\draw[very thick] (7-0.2,0.5) -- (8+0.2,0.5);
\draw[very thick] (7-0.2,1) -- (8+0.2,1);
\draw[draw=black, very thick, fill=white] (7,-0.2) rectangle ++(1,1.4);
\draw[very thick] (10-0.2,0.5) -- (11+0.2,0.5);
\draw[very thick] (10-0.2,1) -- (11+0.2,1);
\draw[very thick] (10-0.2,1.5) -- (11+0.2,1.5);
\draw[draw=black, very thick, fill=white] (10,-0.2) rectangle ++(1,1.9);
\draw[very thick] (13-0.2,0.5) -- (14+0.2,0.5);
\draw[draw=black, very thick, fill=white] (13,-0.2) rectangle ++(1,0.9);
\draw[very thick, draw=red] (1,0) .. controls (1.5,0) and (1.5,1.5) .. (2,1.5);
\end{tikzpicture}
\caption{An illustration for Lemma \ref{MainLemma}.}
\label{TrappedCase}
\end{figure}
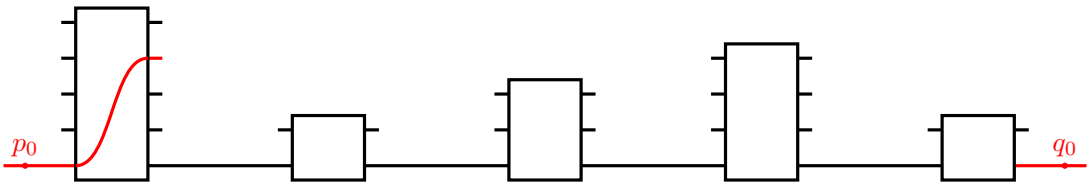

By Lemma \ref{Ground1}, the part of~$\mathrm{LC}(\U; x_1)$ located from~$p_0$ to~$q_0$ lies within the castle~${\rm Cas}(C_1, x_1)$. 
Therefore, since $d>1$, the floor $F_1$ shares at least one crossing with~$F_0$ that lies in the complement of the braid box corresponding to $\a_1$.
Let $j \in \{2,3,\ldots,r\}$ be the minimum number such that $F_1$
intersects~$\a_j$. If $j \geq 3$, then $F_0$ and $F_1$ bound a trap of the castle~${\rm Cas}(C_1, x_1)$. Hence, since ${\rm Cas}(C_1, x_1)$ is appropriate, one has $j=2$. 
Thus, the two circles in $\C$ containing $F_0$ and $F_1$ share at least two arcs in $\A$. This contradicts the fact that $(\C, \A)$ is knitted. The claim is proved.

Similar arguments show that for all arcs $\a \in \{\a_2, \a_3, \ldots, \a_r\}$, 
the index of the ending point of the strand~$s(\a,1)$ is equal to~$1$.
Hence, if~$\mathrm{LC}(\U; x_1)$ intersects an arc $\a \in \A$, then one has~$\a \in \{\a_1, \a_2, \ldots, \a_r\}$. Therefore, the case $i=1$ is complete.

Second, suppose $i=2$. 
Recall that $\U^{1} = \U_1 \backslash \mathrm{LC}(\U_1; x_1)$ is the link diagram obtained from~$\U_{1}$ by deleting the link component $\mathrm{LC}(\U_1; x_1)$.
Let $(\C_1,\A_1)$ and $\p_1$ be the template and the braid placement induced from $(\C, \A)$ and $\p$, respectively, by the transformations that one applied in obtaining $\U_1$ from $\D$ and this link component deletion.
Since the template~$(\C,\A)$ is knitted, $(\C_1, \A_1)$ is knitted too. Therefore, we can apply the same arguments as above. 

The case $i>2$ is similar. The lemma is proved.
\end{proof}

Recall that we aim to prove the inequality~$t^-(\U) \leq t^-(\U^\ast)$.

Let $\a \in \A$. Let $r = (r_1, r_2, \ldots, r_{\left\lVert \a \right\rVert-1}) \in \{1,-1\}^{\left\lVert \a \right\rVert-1}$ be such that~$\p(\a)$ is $r$-homogeneous. Let $i_1,i_2,\ldots,i_m \in \{1,2,\ldots,\left\lVert \a \right\rVert\}$ be such that 
$$1=i_1<i_2< \ldots < i_m=\left\lVert \a \right\rVert$$ 
and for all~$k \in \{1,2,\ldots,m-1\}$, one has $r_{i_k}=r_{i_k+1}=\ldots=r_{i_{k+1}-1}$ and $r_{i_{k+1}-1} \neq r_{i_{k+1}}$. 
By~the definition of~locally $-$twisted braids, for each $k \in \{1,2,\ldots,m-1\}$ with $r_{i_k}=-1$, 
the uniform layer of index $k$ of $\p(\a)$
admits a decomposition of the form~$v_{k,1}v_{k,2}v_{k,3}$ such that both~$v_{k,1}$ and~$v_{k,3}$ represent $\Delta_{i_k,i_{k+1}}^{-1}$ (see~Figure~\ref{ExampleD}).

Let $d := \gamma(\U)$. Recall that the link diagram determined by~$\p_{d}$ is~$\U_{d} = \U$. Given a braid word~$w$, denote by~$|w|$ the length of $w$.

\begin{Lemma}\label{LowerBound}
For each $k \in \{1,2,\ldots,m-1\}$ with $r_{i_k}=-1$, one neither smoothed nor flipped at least $2|v_{k,1}|$ negative crossings of~$[\p(\a)]$ in obtaining $[\p_{d}(\a)]$.
\end{Lemma}
\begin{proof}
Without loss of generality, one has $r_1=-1$. First, suppose $k = 1$.

Let $i \in \{1,2, \ldots, d\}$ be the index such that $\mathrm{LC}(\U; x_i)$ contains the starting point of the first Seifert segment of~$[\p(\a)]$.

The link component $\mathrm{LC}(\U; x_i)$ 
determines a curve $t$ on the braid diagram~$[\p(\a)]$.
It~follows from Lemma \ref{Fragment} that $t$ contains crossings of $[v_{1,1}]$ of the form~$\sigma_1^{-1}, \sigma_2^{-1}, \ldots, \sigma_{i_2-1}^{-1}$.
Lemma \ref{MainLemma} implies that $t$ exits~$[\p(\a)]$ through the ending point of the first Seifert segment of~$[\p(\a)]$. Therefore, $t$ contains at least $i_2-1$ additional crossings of~$[\p(\a)]$ of the form~$\sigma_{i_2-1}^{-1}, \sigma_{i_2-2}^{-1}, \ldots, \sigma_{1}^{-1}$. Therefore, one neither smoothed nor flipped at least $2(i_2-1)$ negative crossings of~$[\p(\a)]$ in obtaining $[\p_{d}(\a)]$.

For each $j \in \{2,3, \ldots, i_2-1\}$, we apply similar arguments concerning the link component of $\U$ that contains the starting point of the index $j$ Seifert segment of~$[\p(\a)]$.
Therefore,
one neither smoothed nor flipped at least 
\begin{align*}
2(i_2-1) + 2(i_2-2) + \ldots + 2 = 2 i_2(i_2-1) = 2|v_{1,1}|
\end{align*}
negative crossings of~$[\p(\a)]$ in obtaining $[\p_{d}(\a)]$. 

For $k>1$ such that $r_{i_k}=-1$, the argument is similar. The lemma is proved.
\end{proof}

Denote by $t^-(\U^\ast; \a)$ (resp.\  $t^-(\U; \a)$) the number of negative crossings of~$[\p(\a)]$ that one smoothed in obtaining $\U^\ast$ (resp.\  $\U$). Denote by $n(\p(\a))$ the number of negative crossings of~$\p(\a)$. Let $J = \{k \in \{1,2,\ldots,m-1\} \mid r_{i_k}=-1\}$.
One has
\begin{align}\label{MRI1}
t^-(\U^\ast; \a) + \sum\limits_{k \in J
} 2|v_{k,1}| = n(\p(\a)).
\end{align}
Lemma \ref{LowerBound} implies
\begin{align}\label{MRI2}
\sum\limits_{k \in J
} 2|v_{k,1}| \leq n(\p(\a)) - t^-(\U; \a).
\end{align}
By combining \eqref{MRI1} and \eqref{MRI2}, one has $t^-(\U; \a) \leq t^-(\U^\ast; \a)$. Therefore,
\begin{align*}
t^-(\U) = \sum_{a \in \A} t^-(\U; \a) \leq \sum_{\a \in \A} t^-(\U^\ast; \a) = t^-(\U^\ast). 
\end{align*}
The second step is complete.

\subsection{Third step}

Let $\U$ be a leaf node of $\T$. Recall that the contribution of $\U$ to the skein polynomial~$\P(\D; a,z)$ is
\begin{align*}
(-1)^{t^-(\U)} z^{t(\U)} a^{\omega(\U) - \omega(\D)} ((a-a^{-1})z^{-1})^{\gamma(\U)-1}.
\end{align*}
Therefore, the highest $a$-degree term that $\U$ contributes to~$\P(\D; a,z)$ is
\begin{align*}
(-1)^{t^-(\U)} z^{t(\U) - \gamma(\U) + 1} a^{\omega(\U) - \omega(\D) + \gamma(\U) - 1}.
\end{align*}
Suppose that $\U$ contributes to the the highest $a$-degree term. Recall that in the first step, we showed that there exits a leaf node $\U^\ast$ of $\T$ that contributes to the the highest $a$-degree term. 
Besides, 
in the second step, we showed that $\gamma(\U) = s(\D)$, $\omega(\U) = 0$, and~$t^-(\U) \leq t^-(\U^\ast)$. Therefore, $\U$ contributes to the term 
\begin{align}\label{finalterm}
(-1)^{t^-(\U)} z^{t(\U^\ast)-s(\D)+1} a^{-\omega(\D) + s(\D)-1}    
\end{align}
if and only if $t(\U) = t(\U^\ast)$. 

Suppose that $t(\U) = t(\U^\ast)$. Denote by $p(\D)$ the number of positive crossings of $\D$. Given a leaf node $\V$ of $\T$, denote by $t^+(\V)$ the number of positive crossings of $\D$ that one smoothed in obtaining~$\V$. In the first step, we showed that $t^+(\U^\ast) = p(\D)$. One has
\begin{align*}
t^-(\U^\ast) = t(\U^\ast) - t^+(\U^\ast) = t(\U) - p(\D) \leq t(\U) - t^+(\U) = t^-(\U).
\end{align*}
Hence,
$t^-(\U) = t^-(\U^\ast)$. Therefore, the contribution of $\U$ to the term \eqref{finalterm} has the same sign~$(-1)^{t^-(\U^\ast)}$ as of $\U^\ast$. Thus, this term does not cancel. This shows that $$E = -\omega(\D) + s(\D)-1.$$ The third step is complete. Theorem \ref{Theorem2} is proved.

\begin{spacing}{1.5}
\end{spacing}

\noindent {\bf Acknowledgements}. The author is grateful to Andrei Malyutin for helpful comments and suggestions.


\begin{thebibliography}{BZH13}
\bibitem[Ad94]{Ad94} C. C. Adams, {\it The Knot Book: An Elementary Introduction to the Mathematical Theory of Knots}, New York: W. H. Freeman and Company, 1994.
\bibitem[B$^+$20]{B20} Yu. S. Belousov, M. V. Karev, A. V. Malyutin, A. Yu. Miller, E. A. Fominykh, {\it Lernaean knots and band surgery}, St. Petersburg Math. J., (to appear).
\bibitem[Cr89]{Cromwell89} P. R. Cromwell, {\it Homogeneous links}, Journal of the London Mathematical Society, 2(3):535–552, 1989.
\bibitem[Cr93]{Cromwell93} P. R. Cromwell, {\it Positive braids are visually prime}, Proc. London Math. Soc. (3) 67 (1993), no. 2, 384–424.
\bibitem[Cr04]{CromwellBook} P. R. Cromwell, {\it Knots and Links}, Cambridge University Press, 2004.
\bibitem[Di04]{AOCN} Y. Diao, {\it The additivity of the crossing number}, J Knot Theory Ramif 13, 7 (2004), 857–866.
\bibitem[DHL19]{BIRAL} Y. Diao, G. Hetyei, P. Liu, {\it The braid index of reduced alternating links}, Math. Proc. Cambridge Philos. Soc. 2019: https://doi.org/10.1017/S0305004118000907.
\bibitem[Dun01]{Dun01} N. Dunfield, {\it A table of boundary slopes of Montesinos knots}, Topology 40 (2001), no. 2, 309–315.
\bibitem[DP13]{DP13} I. A. Dynnikov, M. V. Prasolov, {\it Bypasses for rectangular diagrams. A proof of the Jones conjecture and related questions}, Trudy mosk. mat. obshch. 74 (2013), no. 1, 115-173 (Russian); English transl., Trans. Moscow Math. Soc. 74 (2013), 97–144.
\bibitem[E$^+$92]{WPG} D. B. A. Epstein, J. W. Cannon, D. F. Holt, S. V. F. Levy, M. S. Paterson, W. P. Thurston, {\it Word processing in groups}, Jones and Bartlett Publishers, Boston, MA, 1992.
\bibitem[F$^+$85]{Homfly} P. Freyd, D. Yetter, J. Hoste, W. Lickorish, K. Millett, A. Ocneanu, {\it A new polynomial invariant of knots and links}, Bull. Amer. Math. Soc. (N.S.) 12 (1985) 239–246.
\bibitem[FW87]{FW} J. Franks, R. F. Williams, {\it Braids and the Jones polynomial}, Trans. Amer. Math. Soc., 303(1):97–108, 1987.
\bibitem[FK17]{Feller} P. Feller, D. Krcatovich, {\it On cobordisms between knots, braid index, and the upsilon-invariant}, Math. Ann. 369 (2017), no. 1-2, 301–329. MR 3694648, https://doi.org/10.1007/s00208-017-1519-1.
\bibitem[Gru03]{Gruber} H. Gruber, {\it Estimates for the minimal crossing number}, preprint, 2003.
\bibitem[GM14]{GM} J. Gonz\`{a}lez-Meneses, P. M. G. Manch\`{o}n, {\it Closures of positive braids and the Morton--Franks--Williams inequality}, Topology and its Applications, Volume 174, 2014, Pages 14-24, ISSN 0166-8641, https://doi.org/10.1016/j.topol.2014.06.008.
\bibitem[Kau87]{Kauffman} L. Kauffman, {\it State models and the Jones polynomial}, Topology 26(3) (1987), 395–407.
\bibitem[Kir97]{K} R. Kirby, {\it Problems in low-dimensional topology}, In Geometric Topology (Athens, GA, 1993), 35–473. AMS/IP Stud. Adv. Math. 2(2). Providence, RI: American Mathematical Society, 1997.
\bibitem[Kal09]{Kalman} T. Kalman, {\it Meridian twisting of closed braids and the Homfly polynomial}, Mathematical Proceedings of the Cambridge Philosophical Society, 146(3), 649-660. doi:10.1017/S0305004108002016.
\bibitem[LT88]{LT} W. B. R. Lickorish, M. B. Thistlethwaite, {\it Some links with nontrivial polynomials and their crossing-numbers}, Comment. Math. Helv., 63(1), 1988, 527-539.
\bibitem[Lick97]{Lickorish} W. B. R. Lickorish, {\it An introduction to knot theory}, Graduate Texts in Mathematics, vol. 175, Springer-Verlag, New York, 1997. MR 1472978.
\bibitem[La09]{Lackenby} M. Lackenby, {\it The crossing number of composite knots}, J. Topol. 2, no. 4 (2009): 747–68.
\bibitem[LDH19]{LDH} P. Liu, Y. Diao, G. Hetyei, {\it The Homfly polynomial of links in closed braid form}, Discrete Mathematics, Volume 342, Issue 1, 2019, Pages 190-200, ISSN 0012-365X, https://doi.org/10.1016/j.disc.2018.09.027.
\bibitem[Mal18]{Malyutin} A. V. Malyutin, {\it On the question of genericity of hyperbolic knots}, Int. Math. Res. Not. (2018). https://doi.org/10.1093/imrn/rny220.
\bibitem[Man12]{HomogeneousSeifert} P. M. G. Manch\`{o}n, {\it Homogeneous links and the Seifert matrix}, Pacific J. Math., 255(2):373–392, 2012.
\bibitem[Men84]{Menasco84} W. Menasco, {\it Closed incompressible surfaces in alternating knot and link complements}, Topology 23 (1984) 37-44.
\bibitem[MT93]{FlypingConjecture} W. Menasco, M. Thistlethwaite, {\it A classification of alternating links}, Ann. of Math. 138 (1993) 113–171.
\bibitem[Men19]{Menasco19} W. W. Menasco, {\it Alternating Knots}, preprint, 2019.
\bibitem[Mor86]{Morton} H. R. Morton, {\it Seifert circles and knot polynomials}, Math. Proc. Cambridge Philos. Soc., 99(1):107–109, 1986.
\bibitem[Mur87]{M} K. Murasugi, {\it The Jones Polynomial and Classical Conjectures in Knot Theory}, Topology 26 (1987): 187-194.
\bibitem[Mur91]{MurasugiA} K. Murasugi, {\it On the braid index of alternating links}, Trans. Amer. Math. Soc. 326 (1) (1991), 237–260.
\bibitem[Nak20]{Nak20} K. Nakagane, {\it A full-twisting formula for the HOMFLY polynomial}, arXiv:2009.05511v1.
\bibitem[Nak04]{Nakamura} T. Nakamura, {\it Notes on the braid index of closed positive braids}, Topology Appl. 2004, 135: 13–31.
\bibitem[Oz02]{Ozawa91} M. Ozawa, {\it Closed incompressible surfaces in the complements of positive knots}, Comment. Math. Helv., 77: 235–243 (2002).
\bibitem[Oz11]{Ozawa11} M. Ozawa, {\it Essential state surfaces for knots and links}, J. Aust. Math. Soc. 91 (2011) 391–404 MR2900614.
\bibitem[PT87]{PT} J. Przytycki, P. Traczyk, {\it Conway algebras and skein equivalence of links}, Proc. Am. Math. Soc. 100 (1987), 744–748.
\bibitem[St78]{Stallings} J. Stallings, {\it Constructions of fibred knots and links}, Algebraic and geometric topology (Proc. Sympos. Pure Math., XXXII, Amer. Math. Soc.), 2:55–60, 1978.
\bibitem[Sto02]{Stoimenow} A. Stoimenow, {\it On the crossing number of positive knots and braids and braid index criteria of Jones and Morton--Williams--Franks}, Trans. Amer. Math. Soc. 354(10) (2002), 3927–3954.
\bibitem[St13]{Index} A. Stoimenow, {\it On the definition of graph index}, J. Aust. Math. Soc. 94 (2013), 417–429.
\bibitem[Thi87]{Thistlethwaite87} M. Thistlethwaite, {\it A Spanning Tree Expansion of the Jones Polynomial}, Topology 26(3) (1987), 297–309.
\bibitem[Thi88]{Thistlethwaite88} M. B. Thistlethwaite, {\it On the Kauffman polynomial of an adequate link}, Invent. Math. 93, no. 2 (1988): 285–96.
\bibitem[Yam87]{Yamada} S. Yamada, {\it The minimal number of Seifert circles equals the braid index of link}, Invent. Math. 891 (1987), 347–356.
\end{thebibliography}
\end{document}